\documentclass[11pt]{amsart}
\usepackage{graphicx}
\usepackage{latexsym}
\usepackage{amsfonts,amsmath,amssymb}

\def\si{\par\smallskip\noindent}
\def\bi{\par\bigskip\noindent}
\def\pr{\textup{ P\/}}
\def\ex{\textup{E\/}}

\def\eps{\varepsilon}
\def\la{\lambda}

\def\part{\partial}

\newcommand{\beq}{\begin{equation}}
\newcommand{\eeq}{\end{equation}}

\newtheorem{Theorem}{Theorem}[section]
\newtheorem{Lemma}[Theorem]{Lemma}

\newtheorem{Corollary}[Theorem]{Corollary}

\theoremstyle{remark}

\numberwithin{equation}{section}
\date{\today}
\begin{document}

\title[Stable matchings]{On likely solutions of the stable matching problem with unequal numbers of
men and women}

\author{Boris Pittel}
\address{Department of Mathematics, The Ohio State University, Columbus, Ohio 43210, USA}
\email{bgp@math.ohio-state.edu}



\begin{abstract}
Following up a recent work by Ashlagi, Kanoria and Leshno, we study a stable matching problem with unequal numbers of men and  women, and independent uniform preferences. The asymptotic formulas for the expected number of stable matchings, and for the probabilities of one point--concentration
for the range of husbands' total ranks and for the range of wives'  total ranks are obtained.
\end{abstract}

\maketitle

\section{Introduction and main results} Consider the set of $n_1$ men and $n_2$ women facing a problem of selecting a marriage partner. For $n_1=n_2=n$,  a marriage $\mathcal M$ is a matching (bijection) between the two sets.  It is assumed that each man and each woman has his/her preferences for a marriage partner, with no ties allowed. That is, there are given $n$ permutations $\sigma_j$ of the men set and $n$ permutations $\rho_j$ of the women set, each $\sigma_j$ ($\rho_j$ resp.) ordering the women set (the men set resp.) according to the desirability degree of a woman
(a man) as a marriage partner for man $j$ (woman $j$). A marriage is called stable if there is
no unmarried pair (a man, a woman) who prefer each other to their respective partners in the marriage. A classic theorem, due to Gale and Shapley \cite{GalSha}, asserts that, given any system of preferences $\{\rho_j,\sigma_j\}_{j\in [n]}$, there exists at least one stable marriage $\mathcal M$. 

The proof of this theorem is algorithmic.
A bijection is constructed in steps such that at each step every man not currently on hold makes
a proposal to his best choice among women who haven't rejected him before, and the chosen woman either provisionally puts the man on hold or rejects him, based on comparison of him to her  current suitor if she has one already. Since a woman who once gets proposed to always has a man on hold
afterwards, after finally many steps every woman has a suitor, and the resulting bijection turns out to be stable. Of course the roles can be reversed, with women proposing and each man selecting  between the current proponent and a woman whose proposal he currently holds, if there is such a woman.
In general, the two resulting matchings, $\mathcal M_1$ and $\mathcal M_2$ are different, one man-optimal, another woman-optimal. ``Man/woman-optimal'' means that each man/woman is matched with the best woman/
man among all his/her stable women/men, i.e. those who are the man's/woman's partner in at least one stable matching. Strikingly, the man-optimal (woman-optimal) stable matching is woman-pessimal (man-pessimal), meaning that every woman (man) is matched to her/his worst stable husband (wife). The interested reader is encouraged to consult Gusfield and Irving \cite{GusIrv} for
a rich, detailed analysis of the algebraic (lattice) structure of stable matchings set, and a collection of proposal algorithms for
determination of stable matchings in between the two extremal matchings $\mathcal M_1$ and
$\mathcal M_2$.

A decade after the Gale-Shapley paper, McVitie and Wilson \cite{McVWil} developed an alternative, sequential, algorithm in which proposals by one side to another are made one at a time. This procedure delivers the same
matching as the Gale-Shapley algorithm; the overall number of proposals made, say by men to
women, is clearly the total rank of the women in the terminal matching. 

This purely combinatorial, numbers-free, description  
begs for a probabilistic analysis of the problem chosen uniformly at random among all the instances, whose total number is $(n!)^{2n}$. Equivalently the $2n$ preference permutations
$s_j$ and $\sigma_j$ are uniform, and independent. In a pioneering paper \cite{Wil} Wilson reduced
the work of the sequential algorithm to a classic urn scheme (coupon-collector problem) and proved that the expected running time, whence the expected total rank of wives in the man-optimal
matching, is at most $nH_n\sim n\log n$, $H_n=\sum_{j=1}^n 1/j$.

Few years later  Knuth \cite{Knu}, among other results,  found a better upper bound $(n-1)H_n+1$, and established a matching lower bound $nH_n-O(\log^4n)$. He also posed a series of open problems, one of them
on the {\it expected\/} number of the stable matchings. Knuth pointed out that an answer might be found via his formula for the probability $P(N)$ that a generic matching $\mathcal M$ is stable:
\begin{equation}\label{Pn=}
P(n)=\overbrace {\idotsint}^{2n}_{\bold x,\,\bold y\in [0,1]^n}\,\prod_{1\le i\neq j\le n}
(1-x_iy_j)\, d\bold x d\bold y.
\end{equation}
(His proof relied on an inclusion-exclusion formula, and interpretation of each summand as the value of a $2n$-dimensional integral, with the integrand equal to the corresponding summand in the expansion of the integrand in \eqref{Pn=}.) And then the expected value of $S(n)$, the total number of stable matchings, would then be determined from $\ex[S(n)]=n! P(n)$. 

Following Don Knuth's suggestion, in \cite{Pit1} we used the equation \eqref{Pn=} to obtain an asymptotic formula 
\begin{equation}\label{Pnsim}
P(n)=(1+o(1))\frac{e^{-1}n\log n}{n!},
\end{equation}
which implied that $\ex[S(n)]\sim e^{-1}n\log n$. We also found the integral formulas  for $P_k(n)$ ($P_{\ell}(n)$ resp.) the 
probability that the generic matching $\mathcal M$ is stable {\it and\/} that the total man-rank $R(\mathcal M)$
(the total woman-rank $Q(\mathcal M)$ is $\ell$ resp.).
These integral formulas implied that with high probability (w.h.p. from now) 
for each stable matching $\mathcal M$  the ranks $R(\mathcal M)$, $Q(\mathcal M)$ are 
between $(1+o(1))n\log n$
and  $(1+o(1))n^2/\log n$.  It followed, with some work, that w.h.p.
$R(\mathcal M_1)\sim n^2/\log n$, $Q(\mathcal M_1)\sim n\log n$ and 
$R(\mathcal M_2)\sim n\log n$, $Q(\mathcal M_2)\sim n^2/\log n$. In particular, w.h.p. 
$R(\mathcal M_j)Q(\mathcal M_j)$ $\sim n^3$, ($j=1,2$).  In a joint paper with Knuth and Motwani
\cite{KnuMotPit} we used a novel extension of Wilson's proposal algorithm to show that every woman w.h.p.
has at least $(1/2-o(1))\log n$ stable husbands.

Spurred by these results, in \cite{Pit2} we studied the likely behavior of the full random set $\{(R(\mathcal M),Q(\mathcal M))\}$, where $\mathcal M$ runs through all stable matchings for the random instance of $\{\rho_j,\sigma_j\}_{j\in [n]}$. The key ingredient was the more general formula
for $P_{k,\ell}(n)$, the probability that the generic matching $\mathcal M$ is stable and $Q(\mathcal M)=k$, $R(\mathcal M)=\ell$. We also showed that, for a generic woman, the number of stable husbands 
is  normal in the limit, with mean and variance asymptotic to $\log n$.

The key element of the proofs of the integral representations for these probabilities, which also imply
the Knuth formula \eqref{Pn=}, was a refined, background, probability space. Its sample point is a pair of two
$n\times n$ matrices $\bold X=\{X_{i,j}\}$, $\bold Y=\{Y_{i,j}\}$ with all $2n^2$  entries being independent, $[0,1]$-uniform random variables. Reading each row of $\bold X$ and each column of $\bold Y$ in increasing order we recover the independent, uniform preferences of each of $n$
men and of each of $n$ women respectively. And, for instance, the integrand in \eqref{Pn=} turns out equal to the
probability that a generic matching $M$ is stable, {\it conditioned\/} on the values $x_i=X_{i,M(i)}$,
$y_j=Y_{M^{-1}(j),j}$

Using the formula for $P_{k,\ell}(n)$,
we proved a {\it law of hyperbola\/}:
for every $\la\in (0,1/4)$,  quite surely (q.s) $\max_{\mathcal M}|n^{-3}Q(\mathcal M)R(\mathcal M)-1|\le n^{-\la}$;  ``quite surely'' means with probability $1-O(n^{-K})$, for every $K$, a notion
introduced in Knuth, Motwani and Pittel \cite{KnuMotPit}. 

Moreover, q.s. every point on the hyperbolic arc $\{(u,v):\, uv=1;\,u,v\in [n^{-\la},n^{\la}]\}$ is within distance $n^{-(1/4-\la)}$ from
$n^{-3/2}(Q(\mathcal M), R(\mathcal M))$ for some stable matching $\mathcal M$. In particular, q.s. $S(n)\ge n^{1/2-o(1)}$, a significant  improvement of  the logarithmic bound in \cite{KnuMotPit}, but still far below $n\log n$, the asymptotic order of $\ex[S_n]$.

Thus, for a large number of participants,  a typical instance of the preferences $\{\rho_j,\sigma_j\}_{j\in [n]}$ has multiple stable
matchings very nearly obeying the preservation law for the product of the total man-rank and
the total woman-rank. In a way this law is not unlike thermodynamic laws in physics of gases. However those laws are usually of phenomenological nature, while the product law is a rigorous
corollary of the {\it local\/} stability conditions for the random instance of the preferences
$\{\rho_j,\sigma_j\}_{j\in [n]}$.
 
The hyperbola law implied that w.h.p. the minimum value of  $R(\mathcal M)+Q(\mathcal M)$
(by definition attained at an egalitarian marriage $\mathcal M_3$) is asymptotic to $2n^{3/2}$, 
and the worst spouse rank in $\mathcal M_3$ is of order $n^{1/2}\log n$, while the worst spouse
rank  in the
extremal $\mathcal M_1$ and $\mathcal M_2$ is much larger, of order $n/\log n$.

Recently Lennon and Pittel \cite{LenPit} extended the techniques in \cite{Pit1}, \cite{Pit2} to
show that $\ex[S(n)^2]\sim (e^{-2}+0.5e^{-3})(n\log n)^2$. Combined with \eqref{Pnsim}, this
result implied that $S(n)$ is of order $n\log n$ with probability $0.84$, at least. Jointly with Shepp
and Veklerov \cite{Pit3} we proved that, for a fixed $k$, the expected number of women with $k$ stable husbands is asymptotic to $(\log n)^{k+1}/$ $(k-1)!$.

We hope the reader shares our view that the case of the uniform preferences turned out to be  surprisingly
amenable to the asymptotic analysis, and as such it can serve a benchmark for more general models that might be closer to ``real-life'' situations.   

In this paper we will consider a matching model with sets of men and women of different 
cardinalities $n_1$ and $n_2$, say $n_1<n_2$. In this case Gusfield and Iriving \cite{GusIrv}
defined a stable matching as an {\it injection\/} $\mathcal M: [n_1]\to [n_2]$ such that there is {\it no\/}
unmatched pair $(m,w)$, ($m\in [n_1]$, $w\in [n_2]$), meeting a condition:

\si $m$ prefers $w$ to his partner in $\mathcal M$, {\it and\/} if $w$ is matched in $\mathcal M$ then $w$
prefers $m$ to her partner in $\mathcal M$.

\si It was demonstrated in \cite{GusIrv} that, for any preference lists, at least one stable matching 
(injection) exists, and the women set $[n_2]$ is partitioned into two subsets $A_1$ and $A_2$, ($|A_1|=n_1$, $|A_2|=n_2-n_1$) such that 
 the women from $A_1$ are matched in {\it all\/} stable matchings, and the women from $A_2$  in {\it none\/}.  
 
Few years ago Ashlagi, Kanoria and Leshno \cite{AshKanLes} (see Online Appendices A, B 
and C for the proofs) discovered that the mere positivity
of $n_2-n_1$ drastically changes the likely structure of the stable matchings. Let $R(\mathcal M)$ and
$Q(\mathcal M)$ continue to stand for the total rank of husbands and the total rank of wives in a stable matching
$\mathcal M$. Their main result states:

\begin{Theorem}\label{1} (Ashlagi, Kanoria, Leshno, (AKL)) Let $n_1,\,n_2\to\infty$ and $n_2-n_1>0$. For every $\eps>0$, w.h.p. (1) for every two stable matchings $M$ and $M'$ both
$R(M)/R(M')$ and $Q(M)/Q(M')$ are $1+O(\eps)$, uniformly over $M$ and $M'$; (2) denoting
$s(\bold n)=\log\tfrac{n_2}{n_2-n_1}$,
\[
\max_{\mathcal M}Q(\mathcal M)\le (1+\eps) n_2s(\bold n),\quad \min_{\mathcal M}R(M)\ge
\frac{n_1^2}{1+(1+\eps)\tfrac{n_2}{n_1}s(\bold n)}.
\]
(3) the fraction of men and the fraction of women who have multiple stable partners are each no more than $\eps$.
\end{Theorem}

The contrast with the case $n_1=n_2=n$ is stark indeed. There w.h.p. a generic member  has
$(1+o_p(1))\log n$ stable partners  (see \cite{KnuMotPit}, \cite{Pit2}), and the ratios $R(\mathcal M)/R(\mathcal M')$,
$Q(\mathcal M)/Q(\mathcal M')$ reach the values asymptotic to $n/(\log n)^2$ and $(\log n)^2/n$, \cite{Pit2}. As stressed in \cite{AshKanLes}, Theorem \ref{1} implies that, as long as the focus is on the 
global parameters $R(\mathcal M)$ and $Q(\mathcal M)$ (``centralized markets''),  the likely dependence
on which side proposes {\it almost\/} vanishes. In the outline, I learned of Theorem \ref{1} first from Jennifer Chayes \cite{Cha}, and later, with more details, from Gil Kalai's blog \cite{Kal}. 

As a promising sign, the  basic integral identities for the probabilities $P(n)$ (Knuth \cite{Knu}), $P_k(n)$ and
$P_{\ell}(n)$ (\cite {Pit1}, \cite{Pit2}) for the ``$n_1=n_2=n$''  have natural counterparts for the ``unbalanced'' probabilities $P(\bold n)$,
$P_k(\bold n)$ and $P_{\ell}(\bold n)$, $\bold n:=(n_1,n_2)$. Here $P(\bold n)$ is
the probability that a generic injection $\mathcal M: [n_1]\to [n_2]$ is stable;
(2) $P_k(\bold n)$ ($P_{\ell}  (\bold n)$ resp.) is the probability that a generic injection $\mathcal M: [n_1]\to [n_2]$ is stable and  $Q(\mathcal M)=k$, ($R(\mathcal M)=\ell$ resp.). All three formulas are implied by the integral formula
for $P_{k,\ell}(\bold n)$, the probability that $\mathcal M$ is stable and $Q(\mathcal M)=k$, $R(\mathcal M)=\ell$. 
\vfill
\begin{Lemma}\label{extbasic}
\begin{align*}
P(\bold n)&=\!\overbrace {\idotsint}^{2n_1}_{\bold x,\,\bold y\in [0,1]^{n_1}}\prod_{ i\neq j}(1-x_iy_j)\!\prod_h\bar x_h^{n_2-n_1}\,d\bold x d\bold y,\\
P_k(\bold n)&=\!\overbrace {\idotsint}^{2n_1}_{\bold x,\,\bold y\in [0,1]^{n_1}}\!\!
[\xi^{k-n_1}] \prod_{i\neq j}\bigl(1-x_i(1-\xi+\xi y_j)\bigr)\!\prod_h\bar x_h^{n_2-n_1}\,d\bold x d\bold y,\\
P_{\ell}(\bold n)&=\!\overbrace {\idotsint}^{2n_1}_{\bold x,\,\bold y\in [0,1]^{n_1}}\!\!
[\eta^{\ell-n_1}]\prod_{i\neq j}\bigl(1-y_j(1-\eta+\eta x_i)\bigr)\!\prod_h \bar x_h^{n_2-n_1}\,d\bold x d\bold y,\\
P_{k,\ell}(\bold n)&=\!\overbrace {\idotsint}^{2n_1}_{\bold x,\,\bold y\in [0,1]^{n_1}}\![\xi^{k-n_1}\eta^{\ell-n_1}]
\prod_{ i\neq j}(\bar x_i\bar y_j+x_i\bar y_j\xi+\bar x_iy_j\eta)
\prod_h\bar x_h^{n_2-n_1}\,d\bold x d\bold y;
\end{align*}
here $i,\,j,\,h\in [n_1]$ and $\bar x_i=1-x_i$, $\bar y_j=1-y_j$.  
\end{Lemma}

Thus the condition $n_1<n_2$  leads to insertion of the extra factor $\prod_h\bar x_h^{n_2-n_1}$ into
the corresponding integrands for $P(n)$, $P_k(n)$, $P_{\ell}(n)$, and $P_{k,\ell}(n)$ for the
the ``$n_1=n_2=n$'' case in \cite{Pit2}.\\

Using the formula for $P(\bold n)$, and the fact that  
$\ex[S(\bold n)]=\binom{n_2}{n_1}n_1!\,P(\bold n)$, we will prove
\begin{Theorem}\label{E(S(n))sim} If $n_2>n_1\to\infty$ then
\[
\ex[S(\bold n)]\sim \frac{n_1\exp\Bigl(-\tfrac{e^{s(\bold n)}-1-s(\bold n)}{e^{s(\bold n)}-1}\Bigr)}{(n_2-n_1)s(\bold n)},\qquad s(\bold n):=\log\frac{n_2}{n_2-n_1}.
\]
Consequently, if $n_2\gg n_2-n_1>0$, i.e. $n_2/n_1\to 1$ and $s(\bold n)\to\infty$,  then
\[
\ex\bigl[S(\bold n)] \sim e^{-1}\frac{n_1}{(n_2-n_1)\log n_1},
\]
and if $\tfrac{n_2}{n_1}\to\infty$, then $\ex[S(\bold n)]\to 1$. Finally if  $\lim \tfrac{n_2}{n_1}=c>1$ is finite, then
\[
\lim \ex\bigl[S(\bold n)\bigr]=e^{-1}\frac{\lambda(c)}{\log\lambda(c)},\quad \lambda(c)=\bigl(\tfrac{c}{c-1}\bigr)^{c-1}.
\]
\end{Theorem} 
{\bf Note.\/} Recall that for $n_2=n_1$ we had proved that $\ex\bigl[S(\bold n)\bigr]\sim e^{-1} n_1\log n_1$. Thus increasing the cardinality of one of the sides just by $1$ reduces the asymptotic expected number of stable matchings by the factor of $\log^2 n_1$, but  the resulting number is a sizable $e^{-1}\tfrac{n_1}{\log n_1}$. 

We conjecture that  for $n_2/n_1\to 1$, i.e. when $\ex\bigl[S(\bold n)\bigr]\to\infty$,
the second order moment $\ex\bigl[S^2(\bold n)\bigr]=O\bigl(\ex^2\bigl[S(\bold n)\bigr]\bigr)$, so that,
with a positive limiting probability, $S(\bold n)=\theta\bigl(\ex\bigl[S(\bold n)\bigr]\bigr)$, in a complete
analogy with the case $n_1=n_2$, see \cite{LenPit}.

On the other hand, once the difference $n_2-n_1$ becomes  comparable to  the smaller cardinality, the limiting expected number is finite--still above $1$ as it should be--implying that w.h.p. there are ``just a few'' stable matchings. Finally, if $n_2/n_1\to\infty$ then $\lim \ex\bigl[S(\bold n)\bigr]=1$, implying that $\lim \pr(S(\bold n)=1)=1$.

Our next two theorems establish sharp concentration of the likely ranges $\{Q(\mathcal M)\}_{\mathcal M}$ and $\{R(\mathcal M)\}_{\mathcal M}$ of husbands' and wives' ranks around certain deterministic functions of $\bold n$, when $\mathcal M$ runs through the set of all stable matchings. Let $a,\,b,\,c,\,d<1/2$, and 
\begin{align}
\delta(\bold n)&:=\left\{\begin{aligned}
&s(\bold n)^{-b},&&\text{if }s(\bold n)\to\infty,\\
&n_1^{-a},&&\text{if }s(\bold n)=O(1),\end{aligned}\right.\label{delta(n)=expl}\\
 \mathcal P(\bold n)&:=\left\{\begin{aligned}
&\exp\Bigl(\!- c(n_2-n_1)s(\bold n)^{2(1-b)}\Bigr),&&\text{if }s(\bold n)\to\infty,\\
&\!\exp\Bigl(-\theta\bigl(n_1^{1-2a}\bigr)\Bigr),&&\text{if }s(\bold n)=O(1).
\end{aligned}\right.\notag
\end{align}
\begin{Theorem}\label{Q(M)appr} For $n_2>n_1$ and $n_1$ sufficiently large,
\[
\pr\left(\max_{\mathcal M}\left|\frac{Q(\mathcal M)}{n_2s(\bold n)}-1\right|\ge \delta(\bold n)\right)
\le \mathcal P(\bold n).
\]
\end{Theorem}

Next, introduce 
\[
f(x)=\frac{e^x-1-x}{x(e^x-1)},
\]
and define
\[
\delta^*(\bold n):=\frac{\delta(\bold n)}{s(\bold n)f(s(\bold n))}.
\]
By \eqref{delta(n)=expl}, and $s(\bold n)\ge \tfrac{n_1}{n_2}$,
\[
\delta^*(\bold n)=\left\{\begin{aligned}
&O\bigl(s(\bold n)^{-b}\bigr),&&\text{if }s(\bold n)\to\infty,\\
&O\bigl(n_2n_1^{-1-a}\bigr),&&\text{if }s(\bold n)=O(1),\end{aligned}\right.
\]
where $a,\,b<1/2$. The case $s(\bold n)=O(1)$ forces us to impose the condition $n_2\le n_1^{3/2-d}$, 
$d<1/2$. Under this condition, $\delta^*(\bold n)$ tends to zero for $s(\bold n)=O(1)$ as well,  if we
choose $a<1/2 - d$, which we do.
\begin{Theorem}\label{R(M)appr} Suppose that $n_1<n_2\le n_1^{3/2-d}$ and $n_1$ is sufficiently 
large. Then 
\[
\pr\left(\max_{\mathcal M}\left|\frac{R(\mathcal M)}{n_1^2f(s(\bold n))}-1\right|\ge 1.01\delta^*(\bold n)\right)\le \mathcal P(\bold n).
\]
\end{Theorem}

{\bf Notes.\/}  These two theorems together can be viewed as a quantified analogue of the AKL theorem.
They effectively show how the width of an interval enclosing the scaled concentration point
determines the probability that the full range of the corresponding total rank is contained in this
interval. They demonstrate that while the likely bounds for the range of $Q(\mathcal M)$ in the AKL
theorem are rather sharp in the full range of $n_1,\,n_2$, those for $R(\mathcal M)$ are sharp only in the extreme cases, namely $n_2/n_1\to 1$ and $n_2/n_1\to\infty$.  

That we had  to impose the constraint $n_2\le n_1^{3/2-d}$ came as a surprise.  Most likely, it is an artifact  of our method, and we have no reason to doubt that only the condition
$n_2/n_1\to\infty$ is needed for 
\[
\max_{\mathcal M}\bigl|n_1^{-2} R(\mathcal M)-\tfrac{1}{2}\bigr|\to 0,
\]
in probability. For $n_2\gg n_1^2$, for instance, with probability 
\[
\frac{(n_2)_{n_1}}{n_2^{n_1}}\sim \exp\left(-\frac{n_1^2}{2n_2}\right)\to 1,
\]
 the best marriage candidates for $n_1$ men are all distinct. So w.h.p. in $n_1$ steps of the men-to-women proposal algorithm the men will propose to, and will be accepted by their respective best choices.
By Theorem \ref{E(S(n))sim}, w.h.p. this stable matching $\mathcal M_1$, with $Q(\mathcal M_1)=n_1$, is unique. As for $R(\mathcal M_1)$, it equals the sum of $n_1$ independent, 
$[n_1]$--Uniforms $U_j$, whence in probability
\[
n_1^{-2}R(\mathcal M_1)=n_1^{-2}\sum_{j=1}^{n_1}U_j\to\int_0^1 x\,dx=\tfrac{1}{2}.
\]
It is possible though that appearance of the growth bound for $n_2$ in our argument signals an 
abrupt change in the likely structure of the already unique stable matching, when $n_2$ passes through the threshold $n_1^{3/2}$. Coincidentally, I learned from Yash Kanoria's e-mail
that they also might have tacitly assumed that $n_2$ did not grow too fast with $n_1$.

Finally, we use the powerful result of Irving and Leather \cite{IrvLea} on the lattice of stable matchings,
in combination with an analogue of the formula for $P(\bold n)$ in Lemma \ref{extbasic}, to prove
\begin{Theorem}\label{frctns} Let $m(\bold n)$ and $w(\bold n)$ stand for the fraction of men and
for the fraction of women with more than one stable partner. If $n_2-n_1\ge n_1^{1/2}(\log n_1)^{-\gamma}$, $(\gamma<1)$, then $\lim m(\bold n)=\lim w(\bold n)=0$ in probability.
\end{Theorem}

With some effort, the logarithmic factor (not the power of $n_1$ though) could be improved.  Extension of our approach to the condition  ``$n_2-n_1\ge 1$'' in the AKL theorem is rather problematic.

\section{Proof of Lemma \ref{extbasic}}  First of all, it suffices to consider the injection $M$ such that $M(i)=i$ for all $i\in [n_1]$. 

Introduce the pair of two $n_1\times n_2$ matrices 
$\bold X=\{X_{i,j}\}$, $\bold Y=\{Y_{i,j}\}$ with all $2n_1n_2$ entries being independent, $[0,1]$-uniform random variables. Reading the entries of each row of $\bold X$ and of each column of $\bold Y$ in increasing order, we generate the independent, uniform preferences of each of $n_1$ men and each of $n_2$ women respectively

With probability $1$, $M$ is stable iff 
\begin{equation}\label{Brs=}
\begin{aligned}
&X_{r,s}>X_{r,r},\quad\forall\,r\le n_1,\,s>n_1,\\
&X_{r,s}>X_{r,r}\text{ or }Y_{r,s}>Y_{r,r},\quad\forall\, r,s\le n_1,\,r\neq s.
\end{aligned}
\end{equation}
Call the corresponding events $B_{r,s}$. Crucially, conditioned on the values $x_i=X_{i,i}$, $y_j=
Y_{j,j}$, $i\le n_1$, $j\le n_1$, the $n_1n_2-n_1$ events $B_{r,s}$ are independent, and so
the {\it conditional\/} probability that $M$ is stable equals
\[
\prod_{h\le n_1<j_1}\!\!\!(1-x_h)\prod_{i, j\le n_1,\,i\neq j}\!\!\!(1-x_iy_j)=
\prod_{i, j\le n_1,\,i\neq j}\!\!(1-x_iy_j)\prod_{h\le n_1}\!\bar x_h^{n_2-n_1}.
\]
Integrating this expression over the cube $[0,1]^{2n_1}$, i.e. using the Fubini theorem, we obtain the integral formula for $P(\bold n)$. As for $P_k(\bold n)$, $P_{k,\ell}(\bold n)$ and $P_{k,\ell}(\bold n)$, it suffices to consider
$P_{k,\ell}(\bold n)$. Indeed, the integral representation for $P_k(\bold n)$ will follow by setting
$\eta=1$ and dropping the $[\eta^{\ell-n_1}]$ (extraction) operator. (In fact the formula for $P(\bold n)$ is similarly obtained by setting $\xi=1$, $\eta=1$ and dropping the $[\xi^{k-n_1}\eta^{\ell-n_1}]$
operator.)

Notice that  the wives' and the husbands' total ranks are given by
\begin{align*}
Q(M)&=n_1+\sum_{i=1}^{n_1}\big|\{j\le n_2:X_{i,j}<X_{i,i}\}\bigr|,\\
R(M)&=n_1+\sum_{j=1}^{n_1}\bigl|\{i\le n_1:Y_{i,j}>Y_{j,j}\}\bigr|.
\end{align*}
Indeed, e.g. the number of women whom the man $i$ likes as much as he does his wife $i$ is $1$
plus the number of women $j$ such that $X_{i,j}<X_{i,i}$, whence the formula for $Q(M)$. Of
course, if $M$ is stable, all those women $j$ are among the first $n_1$ women. 
Our task is to compute the probability of the event $\{M\text{ is stable}\}\cap\{Q(M)=k,R(M)=\ell\}$.
Let us determine $P_{k,\ell}(M|\bold x,\bold y)$, the conditional probability of this event given $X_{i,i}=x_i$, $Y_{j,j}=y_j$, ($1\le i,j\le n_1$). Once it is done,  the unconditional $P_{k,\ell}(\bold n)$ is obtained via integrating $P_{k,\ell}(M|\bold x, \bold y)$ over the cube $[0,1]^{2n_1}$.

To this end, we resort to the generating functions  and write
\[
P_{k,\ell}(M|\bold x,\bold y)=[\xi^k\eta^{\ell}]\,\ex\bigl[1_{\{M\text{ is stable}\}}\xi^{Q(M)}\eta^{R(M)}\,|\bold x,\bold y\bigl].
\]
Notice at once that on the event $\{M\text{ is stable}\}$, $Q(M)$ nominally dependent on the
whole $\bold x$ is actually a function of $\{x_i\}_{i\le n_1}$. To determine the underlying polynomial of $\xi,\eta$, it suffices to consider $\xi,\eta\in [0,1]$, in
which case a probabilistic interpretation of this polynomial allows to speed up the otherwise
clumsy derivation.

Go through the unmatched pairs $(i,j)$, i.e. $j\neq i$. Let $j\le n_1$. Whenever $X_{i,j}>X_{i,i} (=x_i)$, {\it mark\/} $(i,j)$ with probability $\xi$; whenever $Y_{i,j}>Y_{j,j} (=y_j)$, {\it color\/}
$(i,j)$ with probability $\eta$, {\it independently\/} of all the previous mark/\linebreak color operations. If both
$X_{i,j}>X_{i,i}$ and $Y_{i,j}>Y_{j,j}$ we perform both mark and color operations on $(i,j)$, independently of each other. On the event $\{M\text{ is stable}\}$ no such pair exists, of course. Then 
\[
\ex\bigl[1_{\{M\text{ is stable}\}}\xi^{Q(M)}\eta^{R(M)}\,|\bold x,\bold y\bigl]=\xi^n\eta^n\pr(\mathcal B|\bold x,\bold y),
\]
where $\mathcal B$ is the event ``$M$ is stable, and all pairs $(i,j)$ eligible for mark/color operation
are marked/colored''. Now 
\[
\mathcal B=\left(\bigcap_{1\le i\neq j\le n_1}\!\!\!\!\mathcal B_{i,j}\right) \bigcap\left(\bigcap_{r\le n_1<s}B_{r,s}\right);
\]
here $B_{r,s}$ is defined in the first line of \eqref{Brs=}, and  $\mathcal B_{i,j}$ is the event
\begin{multline*}
\bigl\{(X_{i,i}<X_{i,j},Y_{j,j}<Y_{i,j})\\
\sqcup (X_{i,i}<X_{i,j}, Y_{j,j}>Y_{i,j},
\,(i,j)\text{ is marked})\\
\sqcup (X_{i,i}>X_{i,j}, Y_{j,j}<Y_{i,j},\,(i,j)\text{ is colored})\bigr\}.
\end{multline*}
Conditioned on the event $\{X_{i,i}=x_i\}_{i\le n_1}\cap \{Y_{j,j}=y_j\}_{j\le n_1}$, 
the $(n_1n_2-n_1)$ events $\mathcal B_{i,j}$, $B_{r,s}$ are all independent, and
\begin{align*}
\pr(\mathcal B_{i,j}|\bold x,\bold y)&=(1-x_i)(1-y_j)+\xi(1-x_i)y_j+\eta x_i(1-y_j)\\
&=\bar x_i\bar y_j+\xi\bar x_iy_j+\eta x_i\bar y_j,\\
\pr(B_{r,s}|\bold x,\bold y)&=1-x_r=\bar x_r.
\end{align*}
Collecting the pieces, and integrating over $(\bold x,\bold y)\in [0,1]^{2n_1}$, we obtain the desired formula for $P_{k,\ell}(\bold n)$.

\section{Proof of Theorem \ref{E(S(n))sim}} \label{P(n)close}
To estimate sharply the $2n_1$-dimensional integral representing $P(\bold n)$ in Lemma \ref{extbasic}, we will use the following facts collected and proved in \cite{Pit1}. Let $X_1,\dots, X_n$ be independent, $[0,1]$-uniform random variables. Denote
\[
\mathcal S_n=\sum_{j=1}^nX_j,\quad \mathcal T_n=\frac{\sum_{j=1}^nX_j^2}{\mathcal S_n^2}.
\]
Also, let $L_1,\dots,L_n$ denote the lengths of the $n$ consecutive subintervals of $[0,1]$ obtained by independently selecting $n-1$ points, each uniformly distributed on $[0,1]$; in particular,
$\sum_jL_j=1$. Define $T_n=\sum_jL_j^2$, $L_n^+=\max_j L_j$.
\begin{Lemma}\label{densities} Let $f_n(s)$, $f_n(s,t)$, $g_n(t)$ denote the
density of $\mathcal S_n$, $(\mathcal S_n,\mathcal T_n)$ and $T_n$ respectively. Then
\begin{equation}\label{fn(s)=}
f_n(s)=\frac{s^{n-1}}{(n-1)!}\pr(L_n^+\le s^{-1});
\end{equation}
in particular
\begin{equation}\label{fn(s)<}
f_n(s)\le \frac{s^{n-1}}{(n-1)!}.
\end{equation}
Furthermore,
\begin{equation}\label{fn(s,t)<}
f_n(s,t)\le \frac{s^{n-1}}{(n-1)!}\,g_n(t).
\end{equation}
\end{Lemma}
We will also need 
\begin{Lemma}\label{Asym:Ln^+,Tn} (1) In probability, 
\[
\lim_{n\to\infty} \tfrac{L_n^+}{n^{-1}\log n}=1,\quad \lim_{n\to\infty}n\, T_n=2,
\]
and (2)
\[
\pr(nT_n\ge 3)=O(n^{-1}).
\]
\end{Lemma}

The relation \eqref{fn(s)=} can be found in Feller \cite{Fel}, Ch. 1, for instance, but the 
inequality \eqref{fn(s,t)<} was new. Both of these relations were proved in \cite{Pit1} by using the
fact that the joint density of $L_1,\dots,L_{n-1}$ is $(n-1)!$ whenever the density is positive. As for
Lemma \ref{Asym:Ln^+,Tn}, (1), its proof was based on a classic equidistribution of $(L_1,\dots,L_n)$
and $\{W_j/\sum_k W_k: 1\le j\le n\}$, where $W_j$ are independent exponentials with parameter $1$, see \cite{Pit1} for the references. This equidistribution delivers the part (2) as follows.
Observe that $\ex[W]=1$, $\ex[W^2]=2$. Choose $a>2$ and $b<1$ such that $a/b^2<3$; for instance $a=2+1/7$ and $b=1-1/7$. Then, denoting $\mathcal W^{(\ell)}=\sum_j W_j^{\ell}$,
\begin{align*}
\pr(nT_n\ge 3)&=\pr\left(\frac{\mathcal W^{(2)}}{(\mathcal W^{(1)})^2}\ge \frac{3}{n}\right)
\le\! \pr\!\left(\mathcal W^{(2)}\ge an\text{ or }\mathcal W^{(1)}<bn\!\right)\\
&\le \!\pr\left(\!\mathcal W^{(2)}\ge an\!\right) +\pr\!\left(\!W^{(1)}<bn\!\right).
\end{align*}
By Chebyshev's inequality, each of these probabilities is of order $O(n^{-1})$, and then so is 
$\pr(nT_n\ge 3)$, which proves the part (2). (We note in passing that this probability is, in fact, much
smaller, certainly below $\exp(-c n^{1/3})$.)

\subsection{Upper bound for $P(\bold n)$}\label{upper}
To bound $P(\bold n)$ from above, we evaluate, asymptotically, the $2n_1$-dimensional integral  in Lemma \ref{extbasic}, integrating first over $\bold y$, and second over $\bold x$. Introduce $s=\sum_ix_i$,
$t=\sum_ix_i^2$, and $s_j=\sum_{i\neq j}x_i=s-x_j$, $t_j=\sum_{i\neq j}x_i^2$. Let $\int_1$,
$\int_2$
denote the contribution to $P(\bold n)$ coming from $(\bold x,\bold y)$ with $\tfrac{t}{s^2}\ge 3/n_1$,
and with $\tfrac{t}{s^2}\le 3/n_1$ respectively. Since $1-\alpha\le e^{-\alpha}$, we have
\begin{align*}
\int_1&\le \overbrace {\idotsint}^{n_1}_{\tfrac{t}{s^2}\ge 3/n_1}\left(\prod_j \int_0^1
e^{-ys_j}\,dy\right)e^{-(n_2-n_1)s}\, d\bold x\\
&=\overbrace {\idotsint}^{n_1}_{\tfrac{t}{s^2}\ge 3/n_1}\left(\prod_j\frac{1-e^{-s_j}}{s_j}\right)\cdot
e^{-(n_2-n_1)s}\,d\bold x.
\end{align*}
Now it is easy to check that
\begin{equation}\label{log,easy}
\left(\log\frac{1-e^{-\xi}}{\xi}\right)' \ge -\min\bigl\{\tfrac{1}{2},\tfrac{1}{\xi}\bigr\};
\end{equation}
therefore
\begin{equation*}
\sum_j\log\frac{1-e^{-s_j}}{s_j}\le n_1\log\frac{1-e^{-s}}{s}+2,\quad\forall\,s>0.
\end{equation*}
Applying Lemma \ref{densities}, \eqref{fn(s,t)<}, we obtain then
\begin{align}
\int_1&\le e^2\frac{\pr\bigl(T_{n_1}\ge 3/n_1\bigr)}{(n_1-1)!}\int_0^{\infty} s^{-1} I(s)\,ds,\notag\\
I(s)&:=e^{-(n_2-n_1)s}\, (1-e^{-s})^{n_1}.\label{I(s)=}
\end{align}
By Lemma \ref{Asym:Ln^+,Tn}, (2), the front probability is $O(n_1^{-1})$,
at most. In addition
\begin{equation}\label{(1-e^{-s})^{n1-1}}
\begin{aligned}
&\int_0^{\infty} s^{-1} I(s)\,ds\le \int_0^{\infty} e^{-(n_2-n_1)s}\,(1-e^{-s})^{n_1-1}\,ds\\
&=\int_0^1 z^{n_2-n_1-1}(1-z)^{n_1-1}\,dz=\frac{1}{(n_2-1)\binom{n_2-2}{n_1-1}}\\
&=\frac{n_2}{n_1}\cdot\frac{1}{(n_2-n_1)\binom{n_2}{n_1}}.
\end{aligned}
\end{equation}
Therefore
\begin{equation}\label{int1<expl}
\int_1\le_b \frac{1}{n_1\,n_1!}\cdot\frac{1}{(n_2-n_1)\binom{n_2}{n_1}}.
\end{equation}
(We use $F_n\le_b G_n$ to indicate that $F_n=O(G_n)$ when the expression for $G_n$ is too bulky.)

Turn to $\int_2$, i.e. the contribution to $P(\bold n)$ from $\bold x$ with $ts^{-2}<3n^{-1}$.  Using $1-\alpha\le e^{-\alpha-\alpha^2/2}$ this time, we need to bound, sharply,
\[
\prod_j\int_0^1\exp\bigl(-ys_j-y^2t_j/2\bigr)\,dy.
\]
Introduce a new variable $z=ys_j$ in the $j$-th factor of the product.  Using $s_j\le s$, $t_j\le t$, where appropriate, and integrating by parts once,  we have
\begin{multline*}
\int_0^1\exp\left(-ys_j-\frac{y^2t_j}{2}\right)\,dy\le \frac{1}{s_j}\int_0^{s_j}\!\exp\left(-z-z^2
\frac{t_j}{2s^2}\right)\,dz\\
\le  \frac{1}{s_j}\left(\int_0^s\!\exp\left(-z-z^2\frac{t_j}{2s^2}\right)\,dz-(s-s_j)e^{-s-t/2}\right)\\
=\frac{1}{s_j}\left(\!1-e^{-s-t_j/2}-\frac{t_j}{s^2}\int_0^{s}\!z\exp\left(-z-z^2\frac{t_j}{2s^2}\right)\,dz
-x_je^{-s-t/2}\!\right)\\
\le\frac{1-e^{-s-t/2}}{s_j}\left(\!1-\frac{\tfrac{t_j}{s^2}}{1-e^{-s-t/2}}\int_0^{s}\!z\exp\left(-z-z^2\frac{t_j}{2s^2}\right)\,dz-\frac{x_j}{e^{s+t/2}-1}\!\right)\\
\le\frac{1-e^{-s-t/2}}{s_j}\left(1-\frac{t_j}{s^2}F\bigl(s,\tfrac{t}{s^2}\bigr)-\frac{x_j}{e^{s+t/2}-1}\right)\\
\le\frac{1-e^{-s-t/2}}{s_j}\exp\left(-\frac{t_j}{s^2}F\bigl(s,\tfrac{t}{s^2}\bigr)-\frac{x_j}{e^{s+t/2}-1}\right),
\end{multline*}
where
\begin{equation}\label{F(s,u)=}
F(s,u):=\frac{1}{1-\exp\bigl(-s-s^2\tfrac{u}{2}\bigr)}\int_0^{s}\!z\exp\bigl(-z-z^2\tfrac{u}{2}\bigr)\,dz;
\end{equation}
in particular,
\begin{equation}\label{F(s,0)=}
F(s,0)=1-\frac{s}{e^s-1},\quad F(s,u)=F(s,0)(1+O(u)),\quad u\downarrow 0,
\end{equation}
uniformly for $s>0$.  Next,
\begin{align}
1-e^{-s-t/2}&=(1-e^{-s})\left(1+\frac{1-e^{-t/2}}{e^s-1}\right)\le
(1-e^{-s})\left(1+\frac{t/2}{e^s-1}\right)\notag\\
&\le (1-e^{-s})\exp\left(\frac{s^2}{2(e^s-1)}\,\frac{t}{s^2}\right),\label{e^{-s-t/2}}
\end{align}
so that
\[
\bigl(1-e^{-s-t/2}\bigr)^{n_1}\le (1-e^{-s})^{n_1}\exp\left(\frac{s^2}{2(e^s-1)}\,\frac{n_1t}{s^2}\right),
\]
with the last exponent bounded as $n_1\to\infty$, uniformly for all $\bold x$ in question, i.e.
meeting the constraint  $\tfrac{n_1t}{s^2}\le 3$. Also, as $\sum_j x_j=s$,
\begin{align*}
\prod_j s_j&=s^{n_1}\prod_j\left(1-\frac{x_j}{s}\right)=s^{n_1}\frac{\prod_{j'}\bigl(1-\tfrac{x_{j'}^2}{s^2}\bigr)}{\prod_{j}\bigl(1+\tfrac{x_{j}}{s}\bigr)}\\
&\ge e^{-1}\,s^{n_1} \left(1-\frac{t}{s^2}\right)\ge e^{-1}\,s^{n_1} \left(1-\frac{3}{n_1}\right).
\end{align*}
Collecting the bounds, and using $\sum_j t_j=(n_1-1)t$,
we get 
\begin{multline}\label{bound1}
\prod_j \int_0^1\exp\left(\!-ys_j-\frac{y^2t_j}{2}\right)\,dy\\
\lesssim e\left(\!\frac{1-e^{-s}}{s}\!\right)^{n_1}
\exp\left(\!-\frac{(n_1-1)t}{s^2}F(s,0)-\frac{s}{e^{s+t/2}-1}+\frac{s^2}{2(e^s-1)}\frac{n_1t}{s^2}\right).
\end{multline}
(We use $F_n\lesssim G_n$ to mean that $\limsup F_n/G_n\le 1$.)
And for the remaining factor $\prod_h \bar x_h$ in the integrand for $P(\bold n)$ we have 
\begin{align}
\prod_h \bar x_h^{n_2-n_1}&\le\exp\left(-(n_2-n_1)\sum_hx_h -\frac{n_2-n_1}{2}\sum_h x_h^2\right)\notag\\
&=\exp\left(-(n_2-n_1)s -\frac{(n_2-n_1)s^2}{2} \frac{t}{s^2}\right).\label{bound2}
\end{align}
Now, $s$ and $t$  are the generic values of the random variables $\mathcal S_{n_1}=\sum_jX_j$
and $\sum_j X_j^2$ respectively. By \eqref{fn(s,t)<}, the joint density of $\mathcal S_{n_1}$ and $\mathcal T_{n_1}=\mathcal S_{n_1}^{-2}\sum_j X_j^2$ is bounded above by $\tfrac{s^{n_1-1}}{(n_1-1)!}$ times the density of $T_{n_1}=\sum_j L_j^2$.  So integrating the product of the
bounds in \eqref{bound1} and \eqref{bound2} over $\bold x\in [0,1]^{n_1}$, meeting the
constraint $\tfrac{t}{s^2}\le 3$, we obtain
\begin{equation}\label{Intleq,<n1}
\begin{aligned}
&\qquad\int_2\le \frac{e}{(n_1-1)!}\int\limits_{0}^{n_1}\frac{I(s)}{s}\,\ex[U_{n_1}(s)]\,ds;\\
&U_{n_1}(s):=1_{\bigl\{T_{n_1}\le\tfrac{3}{n_1}\bigr\}}\exp\!\left(\!-\frac{(n_2-n_1)s^2}{2n_1}n_1T_{n_1}
\right.\\
&\left.-(n_1-1)T_{n_1}F(s,0)
-\frac{s}{\exp\bigl(s+s^2T_{n_1}/2\bigr)-1}+\frac{s^2}{2(e^s-1)}\,n_1T_{n_1}\!\!\right)\!;
\end{aligned}
\end{equation}
recall that $I(s)$ was defined in \eqref{I(s)=}.

The function $I(s)$ attains its  maximum at  $s(\bold n)=\log\tfrac{n_2}{n_2-n_1}$. We anticipate, but will have to prove, that the dominant contribution to the integral in \eqref{Intleq,<n1} comes from  
\begin{equation}\label{mathcalI=}
s\in \mathcal I(\bold n):=\bigl[(1-\delta(\bold n))s(\bold n), (1+\delta(\bold n))s(\bold n)\bigr], 
\end{equation}
for some $\delta(\bold n)\to 0$.

To this end, we need to have a close look at $\ex\bigl[U_{n_1}(s)\bigr]$. First, throwing out the
negative summands from the exponent in the expression \eqref{Intleq,<n1} for $U_{n_1}(s)$,
we have
\begin{align*}
U_{n_1}(s)&\le 1_{\bigl\{T_{n_1}\le\tfrac{3}{n_1}\bigr\}}\cdot\sup_{s>0}\exp\left(\frac{s^2}{2(e^s-1)}n_1T_{n_1}\right)\\
&\le \sup_{s>0}\exp\left(\frac{3s^2}{2(e^s-1)}\right)<\infty.
\end{align*}
Consider $s\in \mathcal I(\bold n)$. For the first term 
in the exponent for $U_{n_1}(s)$,
\begin{align*}
\frac{(n_2-n_1)s^2}{2}T_{n_1}&=\frac{(n_2-n_1)s^2(\bold n)}{2n_1}n_1T_{n_1}
+n_1T_{n_1}O\left(\delta(\bold n)\frac{s^2(\bold n)(n_2-n_1)}{n_1}\right)\\
&=\frac{(n_2-n_1)s^2(\bold n)}{2n_1}(n_1T_{n_1})+o(1),
\end{align*}
because 
\[
\frac{s^2(\bold n)(n_2-n_1)}{n_1}=\frac{(n_2-n_1)\log^2\tfrac{n_2}{n_2-n_1}}{n_1}=O(1)
\]
for $n_2>n_1\to\infty$ and $n_1T_{n_1}\le 3$. Furthermore
\begin{align*}
F(s,0)&=F(s(\bold n),0)+O\Bigl(\delta(\bold n)s(\bold n)e^{-s(\bold n)(1-\delta(\bold n))}\Bigr)
=F(s(\bold n),0)+o(1).
\end{align*}
For $T_{n_1}\le \tfrac{3}{n_1}$, two remaining terms in the exponent for $U_{n_1}(s)$ are bounded for all $s>0$ as well.
So, using $n_1T_{n_1}\to 2$ in probability, by  the bounded convergence theorem,
we obtain:
uniformly for $s\in \mathcal I(\bold n)$,
\begin{align*}
\lim \ex\bigl[U_{n_1}(s)\bigr]&=\exp\left(-\lim \frac{(n_2-n_1)s(\bold n)^2}{n_1}-2\lim F(s(\bold n),0)\right.\\
&\qquad\qquad\qquad+\left.\lim\frac{s^2(\bold n)-s(\bold n)}{e^{s(\bold n)}-1}\right),
\end{align*}
where $n_2>n_1\to \infty$ in such a way that all three limits on the RHS exist. Notice that,
since $s(\bold n)=\log\tfrac{n_2}{n_2-n_1}$, the
limits are bounded by absolute constants whether $\limsup s(\bold n)$ is finite or infinite.
Effectively this means that  $\ex\bigl[U_{n_1}(s)\bigr]\sim e^{-h(\bold n)}$, uniformly for $s\in \mathcal I(\bold n)$, where, by \eqref{F(s,0)=},
\begin{equation}\label{h(n)=}
\begin{aligned}
h(\bold n)&:=\frac{(n_2-n_1)s^2(\bold n)}{n_1}+2F(s(\bold n),0)+\frac{s(\bold n)-s^2(\bold n)}{\exp(s(\bold n))-1}\\
&= 2+\frac{(n_2-n_1)s^2(\bold n)}{n_1}-\frac{s(\bold n)+s^2(\bold n)}{e^{s(\bold n)}-1}\\
&=2-\frac{s(\bold n)}{e^{s(\bold n)}-1}.
\end{aligned}
\end{equation}
provided that  $n_2>n_1\to\infty$.  Consequently
\begin{equation}\label{IntI(s)}
\begin{aligned}
\int\limits_{s\in \mathcal I(\bold n)}\!\!\!\frac{I(s)}{s}\,\ex\bigl[U_{n_1}(s)\bigr]\,ds
&\sim\frac{e^{-h(\bold n)}}{s(\bold n)}
 \int\limits_{s\in \mathcal I(\bold n)}\!\!\!I(s)\,ds.
\end{aligned}
\end{equation}
The next step is to evaluate, asymptotically, the integral in \eqref{IntI(s)}.

Substituting $z=e^{-s}$, we rewrite
\begin{align*}
\int\limits_{s\in\mathcal I(\bold n)}\!\!\! I(s)\,ds&=\int\limits_{z\in [z_1,z_2]}\!\!\!K(z)\,dz,
\quad K(z):=z^{n_2-n_1-1}(1-z)^{n_1},\\
 z_{1,2}&:=\left(\!\frac{n_2-n_1}{n_2}\!\right)^{1\pm\delta(\bold n)}.
\end{align*}
Now
\begin{equation}\label{intK(z)=}
\int\limits_{z\in [0,1]}K(z)\,dz=\frac{1}{(n_2-n_1)\binom{n_2}{n_1}},
\end{equation}
and we need to bound the contributions of the two tail integrals, over $[0,z_1]$ and $[z_2,1]$. 

Consider first the case  $n_2-n_1\ll n_2$, i.e. $s(\bold n)\to\infty$. We have
\begin{equation}\label{weak}
\begin{aligned}
\int\limits_{z\in [0,z_1]}\!\!\!\!K(z)\,dz&\le \int\limits_{z\in [0,z_1]}\!\!z^{n_2-n_1-1}\,dz\\
&=\frac{z_1^{n_2-n_1}}{n_2-n_1}=\left(\frac{n_2-n_1}{n_2}\right)^{(1+\delta(\bold n))(n_2-n_1)}.
\end{aligned}
\end{equation}
So, using \eqref{intK(z)=} and $\binom{b}{a}\le \left(\tfrac{eb}{a}\right)^a$, $m^{1/m}\le e^{1/e}$,
we have 
\begin{equation}\label{[0,z1]ratio}
\frac{\int_0^{z_1} K(z)\,dz}{\int_0^1 K(z)\,dz} \le\eps_1(\bold n):=\exp\bigl(-0.99\delta(\bold n)s(\bold n)(n_2-n_1)\bigr)\to 0,
\end{equation}
if $\delta(\bold n)=s^{-b} (\bold n)$, $b<1$; ($\delta(\bold n)\to 0$ obviously). Next, since
for such  $\delta(\bold n)$ 
\[
n_1z_2=n_1\!\left(\!\frac{n_2-n_1}{n_2}\!\right)^{1-\delta(\bold n)}=\frac{n_1(n_2-n_1)}{n_2}
\left(\!\frac{n_2}{n_2-n_1}\!\right)^{\delta(\bold n)}\to\infty,
\]
we bound 
\begin{align*}
\int_{z_2}^1 K(z)\,dz&\le\int_{z_2}^{\infty} z^{n_2-n_1-1}\,e^{-n_1z}\,dz\le\frac{1}{n_1^{n_2-n_1}}\int_{n_1z_2}^{\infty} y^{n_2-n_1}\,e^{-y}\,dy\\
&=O\bigl(z_2^{n_2-n_1}e^{-n_1z_2}\bigr);
\end{align*}
the last estimate follows from
\[
\max_{y\ge n_1z_2}\frac{d \log(y^{n_2-n_1}\,e^{-y})}{dy}=-1+\frac{n_2-n_1}{n_1z_2}\to -1.
\]
It follows easily that
\begin{equation}\label{[z2,1]ratio}
\begin{aligned}
\frac{\int_{z_2}^1 K(z)\,dz}{\int_0^1 K(z)\,dz}&\le \exp\left(-0.99(n_2-n_1)\bigl(\tfrac{n_2}{n_2-n_1}\bigr)^{\delta(\bold n)}\right)\\
&=\exp\Bigl(-0.99(n_2-n_1)e^{\delta(\bold n)s(\bold n)}\Bigr)=:\eps_2(\bold n).
\end{aligned}
\end{equation}
By \eqref{[0,z1]ratio} and \eqref{[z2,1]ratio}, if $\delta(\bold n)s(\bold n)\to\infty$ then
\begin{equation}\label{znotinz1,z2}
\frac{\int_{z\notin [z_1,z_2]}K(z)\,dz}{\int_0^1 K(z)\,dz}\le\eps(\bold n):=\max\{\eps_1(\bold n),\eps_2(\bold n)\}=
\eps_1(\bold n)\to 0.
\end{equation}
So, by \eqref{intK(z)=},  for $n_2\gg n_2-n_1$ we have
\begin{equation}\label{equiv,error,n2>>}
\int\limits_{s\in \mathcal I(\bold n)}\!\!I(s)\,ds=\int\limits_{z\in [z_1,z_2]}\!\!K(z)\,dz\sim\int\limits_{z\in [0,1]}\!\!K(z)\,dz=\frac{1}{(n_2-n_1)\binom{n_2}{n_1}},
\end{equation}
and the {\it relative\/} contribution of the tail $s\notin \mathcal I(\bold n)$ is of order $\eps(\bold n)$.

Let us prove that, for an appropriate $\delta(\bold n)\to 0$, (1) the equation \eqref{equiv,error,n2>>} holds also when $n_2-n_1=\theta(n_2)$ and (2) the relative weight
of $s\notin \mathcal I(\bold n)$ is at most some alternative $\eps(\bold n)\to 0$.  The integrand
$K(z)$ attains its maximum at $z^*=\tfrac{n_2-n_1-1}{n_2-1}$, and
\begin{equation}\label{K=binom}
K(z^*)=\frac{(n_2-n_1-1)^{n_2-n_1-1}n_1^{n_1}}{(n_2-1)^{n_2-1}}=O\left(n_1^{-1/2}\binom{n_2}{n_1}^{-1}\right).
\end{equation}
Furthermore
\begin{align*}
z_{1,2}-z^*&=\left(\frac{n_2-n_1}{n_2}\right)^{1\pm \delta}-\frac{n_2-n_1-1}{n_2-1}\\
&=\mp\,\theta(\delta n_1/n_2)\bigl(1+O((\delta n_1)^{-1}+\delta n_1/n_2)\bigr);
\end{align*}
the remainder term is $o(1)$ if, in addition to $\delta(\bold n)\to 0$, we impose
the condition $\delta(\bold n)n_1\to\infty$. Simple calculus shows then that
\begin{align*}
\log K(z_{1,2})=\log K(z^*)-\theta(\delta^2n_1),\quad
\left.\frac{d \log K(z)}{dz}\right|_{z=z_{1,2}}=\pm\,\theta(\delta n_2).
\end{align*}
Since $\log K(z)$ is concave, we have then
\begin{align*}
\log K(z)& \le \log K(z_1)+\theta(\delta n_2)(z-z_1),\quad z\le z_1,\\
\log K(z)& \le \log K(z_2)-\theta(\delta n_2)(z-z_2),\quad z\ge z_2.
\end{align*}
Therefore
\begin{equation}\label{znotinz1,z2}
\int\limits_{z\notin [z_1, z_2]}K(z)\,dz\le K(z^*)\,\frac{\exp(-\theta(\delta^2(\bold n)n_1))}{\theta(\delta(\bold n)n_2)}.
\end{equation}
Using  \eqref{K=binom} and \eqref{znotinz1,z2},  we see that the ratio of the upper bound above to 
the integral in \eqref{intK(z)=} is of order
\begin{equation}\label{eps(n),2}
\eps(\bold n):=\frac{\exp\bigl(-\theta(\delta^2(\bold n)n_1)\bigr)}{\delta(\bold n)n_1^{1/2}},
\end{equation}
which tends to zero if, for instance, $\delta(\bold n)=n_1^{-a}$, $a<1/2$. So indeed the equation 
\eqref{equiv,error,n2>>} holds for $\delta(\bold n)=n_1^{-a}$, $a<1/2$, and the relative contribution of the tail $s\notin\mathcal I(\bold n)$ is at most this $\eps(\bold n)$. 

Invoking \eqref{IntI(s)}, we see that if $n_2>n_1\to\infty$, then for  $\mathcal I(\bold n)$ defined in
\eqref{mathcalI=}
\begin{equation}\label{intI(n)sim}
\int\limits_{s\in \mathcal I(\bold n)}\frac{I(s)}{s}\,\ex\bigl[U_{n_1}(s)\bigr]\,ds\sim \frac{e^{-h(\bold n)}}
{s(\bold n)(n_2-n_1)\binom{n_2}{n_1}},
\end{equation}
if, picking $a<1/2$ and $b<1$, we define
\begin{equation}\label{delta(n)=}
\delta(\bold n):=\left\{\begin{aligned}
&s(\bold n)^{-b},&&\text{if }s(\bold n)\to\infty,\\
&n_1^{-a},&&\text{if }s(\bold n)=O(1).\end{aligned}\right.
\end{equation}
Finally, since $\ex\bigl[U_{n_1}(s)\bigr]=O(1)$ uniformly for $s>0$, and $\tfrac{1-e^{-s}}{s}<1$,
we  get
\[
\int\limits_{s\notin \mathcal I(\bold n)}\frac{I(s)}{s}\,\ex\bigl[U_{n_1}(s)\bigr]\,ds\le_b \int\limits_{z\notin [z_1,z_2]}
z^{n_2-n_1-1}(1-z)^{n_1-1}\,dz.
\]
Like the integrals of $K(z)$, (for $\eps(\bold n)$ defined, correspondingly, for $n_2\gg
n_2-n_1$ and $n_2-n_1=\theta(n_2) $), the RHS integral is of order
\begin{multline}\label{intnewK<}
\eps(\bold n)\int_0^1z^{n_2-n_1-1}(1-z)^{n_1-1}\,dz=\eps(\bold n)\cdot\frac{n_2/n_1}{(n_2-n_1)\binom{n_2}{n_1}}\\
\le_b\eps(\bold n)\frac{n_2}{n_1}\log\frac{n_2}{n_2-n_1}\int_{s\in \mathcal I(\bold n)}\frac{I(s)}{s}\,\ex\bigl[U_{n_1}(s)\bigr]\,ds.
\end{multline}
For $n_2\gg n_2-n_1$, i.e. $s(\bold n)\to\infty$, the outside factor is asymptotic to $\eps(\bold n)s(\bold n)$. So, by definition of $\eps(\bold n)$ in \eqref{znotinz1,z2}, and $\delta(\bold n)$ in 
\eqref{delta(n)=}, this factor is of order 
\begin{equation}\label{factor,n2>>n2-n1}
s(\bold n)\exp\bigl(-0.5 s(\bold n)^{1-b}\bigr)=\exp\bigl(-\theta(s(\bold n)^{1-b})\bigr).
\end{equation}
For $n_2-n_1=\theta(n_2)$, by definition of $\eps(\bold n)$ in \eqref{eps(n),2}, and $\delta(\bold n)$ in \eqref{delta(n)=}, the factor is of order
\begin{equation}\label{factor,n2not>>n2-n1}
\eps (\bold n)=\exp\bigl(-\theta(n_1^{1-2a})\bigr). 
\end{equation}

We conclude that if $n_2>n_1\to\infty$ then the bound \eqref{Intleq,<n1} becomes 
\begin{equation}\label{Int2lesssim}
\int_2\lesssim \frac{e^{1-h(\bold n)}}{(n_1-1)!(n_2-n_1)s(\bold n)}\binom{n_2}{n_1}^{-1},
\end{equation}
where $s(\bold n)=\log\tfrac{n_2}{n_2-n_1}$, and $h(\bold n)$ is defined in \eqref{h(n)=}. In
particular 
\[
1-h(n)=-\frac{e^{s(\bold n)}-1-s(\bold n)}{e^{s(\bold n)}-1}.
\]

Combining \eqref{Int2lesssim} and \eqref{int1<expl}, and using $s(\bold n)=O(\log n_1)$,  we obtain
\begin{equation}\label{P(n)lesssim}
P(\bold n)=\int_1+\int_2\lesssim \frac{e^{1-h(\bold n)}}{(n_1-1)!(n_2-n_1)s(\bold n)}\binom{n_2}{n_1}^{-1}.
\end{equation}
Moreover, the contribution to $P(\bold n)$ coming from $s\notin\mathcal I(\bold n)$ and scaled by 
the RHS in \eqref{P(n)lesssim} is of order given by \eqref{factor,n2>>n2-n1} and \eqref{factor,n2not>>n2-n1}
for $n_2\gg n_2-n_1$ and $n_2\not\gg n_2-n_1$, respectively.\\

{\bf Note.\/} Observe that for $n_2/n_1\to\infty$ we have $s(\bold n)\sim\tfrac{n_1}{n_2}\to 0$, so 
$h(\bold n)\to 1$, and
\[
P(\bold n)\lesssim \left(\binom{n_2}{n_1}n_1!\right)^{-1},
\]
meaning that
\[
\ex\bigl[S(\bold n)\bigr] =\binom{n_2}{n_1}n_1!\cdot P(\bold n)\lesssim 1.
\]
Since $S(\bold n)\ge 1$, we see that $\pr\bigl(S(\bold n)=1\bigr)\to 1$.

\subsection{Lower bound for $P(\bold n)$}
It remains to prove a matching lower bound  for $P(\bold n)$.  For $\eps>0$, let $D=D(\eps)$ be a set of $\bold x=(x_1,\dots,x_{n_1})\ge\bold 0$ defined by the constraints
\begin{align}
&s(\bold n)(1-\delta(\bold n)) \le s \le s(\bold n)(1+\delta(\bold n)),\quad s:=\sum_{j=1}^{n_1} x_j,\label{b1}\\
&s^{-1}x_j\le (1+\eps)\frac{\log n_1}{n_1},\quad 1\le j\le n_1,\label{b2}\\
&(1-\eps)\frac{2}{n_1}\le s^{-2}t\le (1+\eps)\frac{2}{n_1},\quad t:=\sum_{j =1}^{n_1}x_j^2. \label{b3}
\end{align}
Here $\delta(\bold n)$ is defined in \eqref{delta(n)=}; so the constraint \eqref{b1} can be stated as $s\in\mathcal I(\bold n)$, see \eqref{mathcalI=}.  
As $s(\bold n)=O(\log n_1)$, the constraints \eqref{b1}, \eqref{b2} imply that
\begin{equation}\label{b4}
\max_j x_j\le (1+\eps)\frac{2 s(\bold n)\log n_1}{n_1}=O(\sigma_{n_1}), \quad \sigma_{n_1}:=n_1^{-1}\log^2 n_1. 
\end{equation}
Since $\sigma_{n_1}\to 0$, $D$ is a subset of the cube $[0,1]^{n_1}$ for $n_1$ large enough.
Likewise the constraints \eqref{b1} and \eqref{b3} imply that
\begin{equation}\label{b5}
t=O\bigl(n_1^{-1}\log^2 n_1\bigr).
\end{equation}

Clearly $P(\bold n)\ge P(\bold n,\eps)$ where $P(\bold n,\eps)$ is the contribution to the integral in the formula for $P(\bold n)$,
(see Lemma \ref{extbasic}), coming from $D$, i.e.
\[
P(\bold n,\eps)=\overbrace {\idotsint}^{n_1}_{\bold x\in D}
\left(\prod_{j=1}^{n_1}\int_0^1\prod_{i\neq j}(1-x_iy_j)\,dy_j\!\!\right)\!\prod_{h=1}^{n_1}\bar x_h^{n_2-n_1}\,d\bold x.
\]
Here, using \eqref{b4}, and then \eqref{b1}, \eqref{b3},
\begin{align*}
\prod_{h=1}^{n_1}\bar x_h^{n_2-n_1}&=\exp\Bigl(-(n_2-n_1)\sum_h\bigl(x_h +\tfrac{x_h^2}{2}+
O(x_h^3)\bigr)\Bigr)\\
&\ge\exp\Bigl(-(n_2-n_1)s +\frac{(n_2-n_1)t}{2}(1+O(\sigma_{n_1}))\Bigr)\\
&=\exp\Bigl(-(n_2-n_1)s+\frac{(n_2-n_1)s^2}{n_1}+O(\eps)\Bigr),
\end{align*}
as the last fraction is uniformly bounded for $s$ meeting \eqref{b1}. Likewise
\begin{align*}
\prod_{i\neq j}(1-x_iy_j)&\ge\prod_{i=1}^{n_1} \exp\Bigl(-y_js_j-y_j^2t\,\frac{1+O(\sigma_{n_1})}{2}\Bigr)\\
&=\prod_{i=1}^{n_1} \exp\Bigl(-y_js_j-y_j^2\,\frac{t}{2}+O(\sigma^2_{n_1})\Bigr),
\end{align*}
where 
\[
s_j:=\sum_{i\neq j}x_i=s-x_j=s\bigl(1-O(n_1^{-1}\log n_1)\bigr).
\]
Since
\[
s_j=s\exp\left(\log\left(1-\frac{x_j}{s}\right)\right)=s\exp\left(-\frac{x_j}{s}+O(x_j^2/s^2)\right),
\]
we have
\begin{equation}\label{prodsj=}
\prod_j s_j=e^{-1}s^{n_1}\exp\bigl(O(ts^{-2})\bigr)=e^{-1}s^{n_1}\bigl(1+O(n_1^{-1})\bigr).
\end{equation}
Therefore, for each $j$,
\begin{equation*}
\begin{aligned}
&I_j(\bold x):=\int_0^1\prod_{i\neq j}(1-x_iy_j)\,dy_j\ge\bigl(1+O(\sigma^2_{n_1})\bigr)\int_0^1\!\!\exp\bigl(\!-ys_j-y^2t/2\bigr)\,dy\\
&\ge\bigl(1+O(\sigma^2_{n_1})\bigr) s_j^{-1}\int_0^{s_j}\!\!\exp\left(\!\!-z-z^2\frac{t}{2s_j^2}\!\right) dz\\
&=\bigl(1+O(\sigma^2_{n_1})\bigr) s_j^{-1}\int_0^{s_j}\!\!\exp\left(\!\!-z-z^2\frac{t}{2s^2}\!\right) dz\\
&\ge\bigl(1+O(\sigma^2_{n_1})\bigr) s_j^{-1}\left(\int_0^{s}\exp\left(\!\!-z-z^2\frac{t}{2s^2}\!\right)
dz -x_je^{-s-t/2}\bigl(1+O(\sigma_{n_1})\bigr)\!\!\right)\!\!.
\end{aligned}
\end{equation*}
Here, as $ts^{-2}=O(n_1^{-1})$, analogously to \eqref{e^{-s-t/2}} we obtain 
\[
1-e^{-s-t/2}=(1-e^{-s})\exp\left(\frac{t}{s^2}\frac{s^2}{2(e^s-1)}+O(n_1^{-2})\right);
\]
also 
\[
\frac{x_je^{-s-t/2}}{1-e^{-s-t/2}}=\frac{x_j}{e^{s+t/2}-1}\le \frac{x_j}{e^s-1}=O(x_j/s)=O(n_1^{-1}\log n_1).
\]
So it follows easily that
\begin{align*}
I_j(\bold x)&\ge \bigl(1+O(\sigma^2_{n_1})\bigr)\left[1-\frac{t/s^2}{1-e^{-s}}\int_0^s ze^{-z}\,dz-
\frac{x_j}{e^s-1}+\frac{t}{s^2}\frac{s^2}{2(e^s-1)}\right].
\end{align*}
Thus, using \eqref{prodsj=},
\begin{align*}
\prod_jI_j(\bold x)&\gtrsim e\!\left(\!\frac{1-e^{-s}}{s}\!\right)^{n_1}\!\exp\!\left(\!\!-\frac{n_1t/s^2}{1-e^{-s}}\int_0^s\!\!ze^{-z}dz
-\frac{s}{e^s-1}+\frac{s^2}{2(e^s-1)}\frac{n_1t}{s^2}\!\right).
\end{align*}
Recalling the constraint \eqref{b3}, we have: uniformly for $\bold x\in D$,
\begin{align}
I(\bold x)&:=\prod_{j=1}^{n_1}\bar x_j^{n_2-n_1}I_j(\bold x)\gtrsim(1+O(\eps))\,e^{1-H(s)}I(s),
\label{I(x)gtrsim}\\
H(s)&:= \frac{n_2-n_1}{n_1}s^2-\frac{2}{1-e^{-s}}\int_0^s ze^{-z}dz +\frac{s^2-s}{e^s-1}\notag\\
&\,\,=\frac{n_2-n_1}{n_1} s^2-2+\frac{s+s^2}{e^s-1}.\notag
\end{align}
And we already proved in the subsection \ref{upper} that, uniformly for $s$ satisfying \eqref{b1},
$H(s)=h(\bold n)+o(1)$, with $h(\bold n)$ given by \eqref{h(n)=}.

To lower-bound $P(\bold n,\eps)$ we integrate the RHS of \eqref{I(x)gtrsim}, with $h(\bold n)$ 
instead of $H(s)$, over $D$. To do so, we switch to $n_1$ new variables $u,\,v_1,\dots, v_{n_1-1}$:
\[
u=\sum_{j=1}^{n_1} x_j, \quad v_j=x_js^{-1},\,\, 1\le j\le n_1-1.
\]
Define also $v_{n_1}=x_{n_1}s^{-1}$. Clearly $0\le v_j\le 1$ and $\sum_{j=1}v_j^{n_1}=1$. The
Jacobian of $(x_1,\dots,x_{n_1})$ with respect to $(u,v_1,\dots,v_{n_1-1})$ is $u^{n_1-1}$.
The constraints \eqref{b1}-\eqref{b3} become
\begin{align}
&(1-\delta(\bold n))s(\bold n)\le u\le (1+\delta(\bold n))s(\bold n), \label{b1new}\\
&v_j\le (1+\eps)\frac{\log n_1}{n_1},\,\,(1\le j\le n_1);\quad \sum_{j=1}^{n_1}v_j=1, \label{b2new}\\
&(1-\eps)\frac{2}{n_1}\le \sum_{j=1}^{n_1}v_j^2\le (1+\eps)\frac{2}{n_1}.\label{b3new}
\end{align}
Obviously, but crucially, none of these constraints involves both $u$ and $\bold v$. Therefore
\begin{equation*}
P(\bold n,\eps)\ge \frac{(1+O(\eps))\,e^{1-h(\bold n)}}{(n_1-1)!} \int\limits_{s\in \mathcal I(\bold n)}\!\!\!\!\!u^{-1}I(u)\,du\, \cdot\! \int\limits_{\bold v\in \mathcal D}\!(n_1-1)!\,\prod_{j=1}^{n_1-1}dv_j,
\end{equation*}
where $\mathcal D$ denotes the set of all $\bold v=(v_1,\dots,v_{n_1})$ meeting the constraints
\eqref{b2new} and \eqref{b3new}.
As we mentioned earlier, $(n_1-1)!$ is the joint density of the  subintervals lengths $L_1,\dots, L_{n_1-1}$ in the random partition of the interval $[0,1]$ by $n_1-1$ points chosen uniformly at random. Therefore
\begin{multline*}
\int_{\mathcal D}(n_1-1)!\,d\bold v\\
=\!\!\pr\!\Biggl(\!\!\left\{L^+_{n_1}\le (1+\eps)\frac{\log n_1}{n_1}\right\}
\bigcap\Bigl\{(1-\eps)\frac{2}{n_1}\le \sum_jL_j^2\le (1+\eps)\frac{2}{n_1}\Bigr\}\!\!\Biggr),
\end{multline*}
which tends to $1$, as $n_1\to\infty$. Furthermore, it was proved in Section \ref{upper}, \eqref{intI(n)sim},  that 
\[
\int\limits_{s\in \mathcal I(\bold n)}\!\!\! u^{-1}I(u)\,du\sim \frac{1}{s(\bold n)(n_2-n_1)\binom{n_2}{n_1}}.
\]
Thus, for every $\eps>0$,
\begin{align*}
P(\bold n)&\ge P(\bold n,\eps)\ge 
(1+O(\eps)) \frac{\exp\Bigl(-\tfrac{e^{s(\bold n)}-1-s(\bold n)}{e^{s(\bold n)}-1}\Bigr)}{(n_1-1)!(n_2-n_1)\binom{n_2}{n_1}s(\bold n)},
\end{align*}
implying that
\begin{equation*}
P(\bold n)\gtrsim \frac{\exp\Bigl(-\tfrac{e^{s(\bold n)}-1-s(\bold n)}{e^{s(\bold n)}-1}\Bigr)}{(n_1-1)!(n_2-n_1)\binom{n_2}{n_1}s(\bold n)}.
\end{equation*}

Combining this estimate with \eqref{P(n)lesssim} we have 
\begin{equation}\label{P(n)sim}
P(\bold n)\sim \frac{\exp\Bigl(-\tfrac{e^{s(\bold n)}-1-s(\bold n)}{e^{s(\bold n)}-1}\Bigr)}{(n_1-1)!(n_2-n_1)\binom{n_2}{n_1}s(\bold n)}.
\end{equation}

Since $\ex[S(\bold n)]$, the expected value of $S(\bold n)$, the number of stable matchings,
is $P(\bold n)\binom{n_2}{n_1}n_1!$, we proved Theorem \ref{E(S(n))sim}: if $n_2>n_1\to\infty$ then
\[
\ex[S(\bold n)]\sim \frac{n_1\exp\Bigl(-\tfrac{e^{s(\bold n)}-1-s(\bold n)}{e^{s(\bold n)}-1}\Bigr)}{(n_2-n_1)s(\bold n)},\qquad s(\bold n)=\log\frac{n_2}{n_2-n_1}.
\]

\section{Proof of Theorem \ref{Q(M)appr}}

Recall that $\min_{\mathcal M}Q(\mathcal M)$ is
the wives' rank $Q(\mathcal M_1)$ in the men-optimal stable matching. So $Q(\mathcal M_1)$ is distributed as the 
total number of proposals by men to women. Now, analogously to the balanced case (Wilson
\cite{Wil}), $Q(\mathcal M_1)$ is {\it stochastically dominated\/} by $N$, the number of consecutive random throws of balls, a ball per throw, into $n_2$ boxes till the moment when there are
exactly $n_1$ non-empty boxes.  (The difference between $N$ and $Q(\mathcal M_1)$ is the total
number of {\it redundant\/} proposals made by men to women who had rejected them earlier.)  $N$ is  distributed as the sum of $n_1$ independent
Geometrics with success probabilities $p_j: =\tfrac{n_2-j}{n_2}$, $0\le j\le n_1-1$. So
\[
\ex[N]=\sum_{j=0}^{n_1-1}\frac{1}{p_j}=n_2\bigl(H_{n_2}-H_{n_2-n_1}\bigr)\sim n_2s(\bold n),
\]
$s(\bold n)=\log\tfrac{n_2}{n_2-n_1}$, if $n_2>n_1\to\infty$. It can be proved that $N$ is sharply concentrated around $\ex[N]$. We will not do it,
but instead will use Lemma \ref{extbasic} to prove  that w.h.p. all $Q(\mathcal M)$ are sharply concentrated around $n_2s(\bold n)$.

Let us start with bounding $\min_{\mathcal M}Q(\mathcal M)$ from below.  To this
end, observe first that, for $k\ge n_1$,
\begin{equation}\label{P(Q(M1)<)<}
\pr(\min_{\mathcal M} Q(\mathcal M)\le k)\le \binom{n_2}{n_1}n_1! \sum_{\kappa<k}P_{\kappa}(\bold n); 
\end{equation}
$P_{\kappa}(\bold n)$ is the probability that a generic injection from $[n_1]$ to $[n_2]$ is stable.
We want to show the RHS is vanishing in the limit  for $k=(1-\delta(\bold n))(n_2-1)s(\bold n)$.
Here $\delta(\bold n)$ is defined in \eqref{delta(n)=}.

By Lemma \ref{extbasic},
\[
P_{\kappa}(\bold n)=\!\overbrace {\idotsint}^{2n_1}_{\bold x,\,\bold y\in [0,1]^{n_1}}\!\!
[\xi^{\kappa-n_1}] \prod_{i\neq j}\bigl(1-x_i(1-\xi+\xi y_j)\bigr)\!\prod_h\bar x_h^{n_2-n_1}\,d\bold x\, d\bold y.
\]
This identity is perfectly suited to application of  Chernoff's method. Denoting 
$P^-_{k}(\bold n):=\sum_{\kappa\le k}P_{\kappa}(\bold n)$, we have 
\begin{align}
P^-_{k}(\bold n)&\le 
\!\overbrace {\idotsint}^{2n_1}_{\bold x,\,\bold y\in [0,1]^{n_1}}
\inf\!\Big\{\!\Phi(\xi,\bold x,\bold y): \xi\in (0,1]\!\Bigr\}\,d\bold x\, d\bold y,\label{sum,kappa<Phi}\\
\Phi(\xi,\bold x,\bold y)&:=\xi^{n_1-k}\!\!\!\!\prod_{1\le i\neq j\le n_1}\!\!\!
\bigl(1-x_i(1-\xi+\xi y_j)\bigr)\!\prod_h\bar x_h^{n_2-n_1}\,d\bold x\, d\bold y.
\label{Phi=}
\end{align}
We will not try to determine the best $\xi=\xi(\bold x,\bold y)$, and focus instead on a judicious choice of $\xi$ dependent only on $s=\sum_ix_i$. Using 
\[
1-x_i(1-\xi+\xi y_j)\le \exp\bigl(-x_i(1-\xi+\xi y_j)\bigr),\quad \bar x_h\le e^{-x_h}, 
\]
and integrating over $\bold y\in [0,1]^{n_1}$, we obtain 
\[
P^-_{k}(\bold n)\le\!\overbrace {\idotsint}^{n_1}_{\bold x\in [0,1]^{n_1}}
\inf_{\xi\in (0,1]}\Bigl\{\xi^{n_1-k}\exp\bigl[s(\xi(n_1-1)-n_2+1)\bigr]\prod_{j=1}^{n_1}
\frac{1-e^{-\xi s_j}}{\xi s_j}\Bigr\}d\bold x,
\]
$(s_j=\sum_{i\neq j}x_i)$. Here, by \eqref{log,easy} and $\xi\le 1$,
\[
\prod_{j=1}^{n_1}\frac{1-e^{-\xi s_j}}{\xi s_j}\le e^2 \left(\frac{1-e^{-\xi s}}{\xi s}\right)^{n_1};
\]
so with the product replaced by its bound, the integrand becomes a function of $s$ only.
Applying \eqref{fn(s)<} in Lemma \ref{densities}, we obtain then
\begin{align}
P^-_{k}(\bold n)&=O(1)\frac{1}{(n_1-1)!}\int_0^{n_1}\!\!\inf_{\xi\in (0,1]}\exp\bigl(H(s,\xi)\bigr)\,ds,
\label{Pk^-simple}\\
H(s,\xi)&:=s\bigl[\xi(n_1-1)-n_2+1\bigr]+n_1\log(1-e^{-\xi s})\notag\\
&\quad\,\,-\log s-k\log \xi.\label{H(s,xi)}
\end{align}
We are willing to bound $\inf_{\xi<1}e^{H(s,\xi)}$
by the value of $e^{H(s,\xi)}$ at  
a stationary point of $H(s,\xi)$ considered as a function of $\xi\in [0,1)$, hoping that this will
the minimum point of this function. Now
\[
H_{\xi}(s,\xi)=s(n_1-1)+sn_1\bigl(e^{\xi s}-1)^{-1}-k\xi^{-1}=0,
\]
if $t:=\xi s$ satisfies an equation
\begin{equation}\label{h(x)=}
h(t)=k,\quad h(u):=u\bigl[(n_1-1)+n_1(e^u-1)^{-1}\bigr].
\end{equation}
Since $h(0+)=n_1<k$ and $h(\infty)=\infty$, a root $t$ does exist, and it is unique, since 
$h'(u)\ge h'(0+)=\tfrac{n_1}{2}-1>0$. So ideally we would like to select $\xi(s)=t/s$ for $s\ge t$ and use the fall-back $\xi(s)\equiv 1$ for $s<t$. The technical issue here is necessity to deal with an implicitly defined $t$ as the root of $h(u)=k$. Observe that 
\[
h(s(\bold n))=s(\bold n)\left[(n_1-1)+\frac{n_1}{e^{s(\bold n)}-1}\right]=(n_2-1)s(\bold n)=\frac{k}{1-\delta(\bold n)}.
\]
So let us try our luck with the explicit $t:=(1-\delta(\bold n))s(\bold n)$, as an approximation for that implicit
root, selecting $\xi(s)=t/s$ for $s\ge t$ and $\xi(s)=1$ for $s<t$.

With $t=(1-\delta(\bold n))s(\bold n)$, we have
\begin{equation}\label{[0,t]}
\begin{aligned}
&\int_0^t\exp\bigl(H(s,\xi(s)\bigr)\,ds=\int_0^t\exp\bigl(H(s,1\bigr)\,ds\\
&\quad=\int_0^ts^{-1}e^{-(n_2-n_1)s}(1-e^{-s})^{n_1}\,ds\\
&\quad\le\int_{z_2}^1z^{n_2-n_1-1}(1-z)^{n_1-1}\,dz\quad \Bigl(z_2=e^{-(1-\delta(\bold n))s(\bold n)}\Bigr)\\
&\quad\le_b \eps^*(\bold n)\int_0^1z^{n_2-n_1-1}(1-z)^{n_1-1}\,dz=\eps^*(\bold n)\cdot\frac{n_2/n_1}{(n_2-n_1)\binom{n_2}{n_1}};
\end{aligned}
\end{equation}
Here  
\begin{equation}\label{new,eps(n)=}
\eps^*(\bold n)=\left\{\begin{aligned}
&\eps_2(\bold n)=\exp\left(-0.99(n_2-n_1)e^{\delta(n)s(\bold n)}\right),&&\text{if }s(\bold n)\to\infty,\\
&\eps(\bold n)=\exp\bigl(-\theta(n_1^{1-2a})\bigr),&&\text{if }s(\bold n)=O(1),\end{aligned}\right.
\end{equation}
cf. \eqref{[z2,1]ratio} and \eqref{eps(n),2}. 

Further, for $s\ge t$,
\begin{equation}\label{H(s,xi(s))=full}
\begin{aligned}
H(s,\xi(s))&=t(n_1-1)-s(n_2-1)+n_1\log\bigl(1-e^{-t}\bigr)\\
&\quad-\log s -k\log\frac{t}{s}.
\end{aligned}
\end{equation}
So 
\begin{align*}
\int_t^{\infty}\exp\bigl(H(s,\xi(s)\bigr)\,ds&\le_b\exp\left(t(n_1-1)+n_1\log \bigl(1-e^{-t}\bigr)-k\log t\right)\\
&\times\int_t^{\infty}\exp\bigl(-s(n_2-1)\bigr)s^{k-1}\,ds.
\end{align*}
Here, using $k=t(n_2-1)$,
\begin{align*}
&\int_t^{\infty}\exp\bigl(-s(n_2-1)\bigr)s^{k-1}\,ds\le\int_0^{\infty}\exp\bigl(-s(n_2-1)\bigr)s^{k-1}\,ds\\
&=\frac{(k-1)!}{(n_2-1)^k}\le_b\left(\frac{k}{e(n_2-1)}\right)^{k-1}
\le\exp\bigl((k-1)\log t-t(n_2-1)\bigr),
\end{align*}
Therefore
\begin{equation}\label{t,oo}
\begin{aligned}
&\int_t^{\infty}\exp\bigl(H(s,\xi(s)\bigr)\,ds\le_be^{\psi(t)},\\
&\psi(s):=-s(n_2-n_1)+n_1\log(1-e^{-s}).
\end{aligned}
\end{equation}
Here
\begin{align*}
\psi(s(\bold n))&=-n_2\log n_2+n_1\log n_1+(n_2-n_1)\log(n_2-n_2)\\
&=-\log\left(\frac{n_2^{n_2}}{n_1^{n_1}(n_2-n_1)^{n_2-n_1}}\right)\le \log \binom{n_2}{n_1}^{-1},\\
\psi'(s(\bold n))&=\left. -(n_2-n_1)+\frac{n_1}{e^s-1}\right|_{s=s(\bold n)}=0,\\
\psi''(s)&=-\frac{e^s}{(e^s-1)^2};
\end{align*}
in particular, for $s\in [t,s(\bold n)]$,
\[
\psi ''(s)\le \psi''(s(\bold n))=-\frac{n_2(n_2-n_1)}{n_1}.
\]
Then 
\begin{align*}
\psi(t)&\le \psi(s(\bold n))+\frac{1}{2}\psi''(s(\bold n)) (t-s(\bold n))^2\\
&=\log\binom{n_2}{n_1}^{-1}\!-\,\frac{n_2(n_2-n_1)}{2n_1}\bigl(\delta(\bold n)s(\bold n)\bigr)^2.
\end{align*}
So
\begin{equation}\label{[t,n1]}
\begin{aligned}
&\int_t^{n_1}\exp\bigl(H(s,\xi(s)\bigr)\,ds\le \eps^{**}(\bold n)\binom{n_2}{n_1}^{-1},\\
&\eps^{**}(\bold n)=\exp\left(-\frac{n_2(n_2-n_1)}{2n_1}\bigl(\delta(\bold n)s(\bold n)\bigr)^2\right).
\end{aligned}
\end{equation}
Adding \eqref{[0,t]} and \eqref{[t,n1]}, we obtain 
\begin{equation*}
\int_0^{n_1}\exp\bigl(H(s,\xi(s)\bigr)\,ds\le_b
\bigl[\eps^*(\bold n)\frac{n_2}{n_1(n_2-n_1)}+\eps^{**}(\bold n)\bigr] \binom{n_2}{n_1}^{-1}.
\end{equation*}

(1) If $s(\bold n)\to\infty$,  then $n_2/n_1\to 1$, $\delta(\bold n)=s(\bold n)^{-b}$, $b<1$, and
\begin{align*}
\eps^*(\bold n)&=\exp\Big(-0.99(n_2-n_1)e^{s(\bold n)^{1-b}}\Bigr),\\
\eps^{**}(\bold n)&=\exp\left(-(0.5+o(1))(n_2-n_1)s(\bold n)^{2(1-b)}\right),
\end{align*}
so that  $\eps^*(\bold n)\frac{n_2}{n_1(n_2-n_1)}\ll \eps^{**}(\bold n)$. 
(2) If $s(\bold n)=O(1)$, then $\delta(\bold n)=n_1^{-a}$, $a<1/2$, and 
\[
\eps^*(\bold n)\frac{n_2}{n_1(n_2-n_1)}+\eps^{**}(\bold n)\le\exp\left(-\theta\bigl(n_1^{1-2a}\bigr)\right).
\]
We conclude that the bound \eqref{Pk^-simple} implies
\begin{equation}\label{Pk^-<simple}
\begin{aligned}
P_k^-(\bold n)&\le_b \frac{\hat{\eps}(\bold n)}{(n_1-1)!}\binom{n_2}{n_1}^{-1},\\
\hat{\eps}(\bold n)&:=\left\{\begin{aligned}
&\exp\left(-0.49(n_2-n_1)s(\bold n)^{2(1-b)}\right),&&\text{if }s(\bold n)\to\infty,\\
&\exp\left(-\theta\bigl(n_1^{1-2a}\bigr)\right),&&\text{if }s(\bold n)=O(1);\end{aligned}\right.
\end{aligned}
\end{equation}
here $a\in (0,1/2)$, $b\in (0,1)$. So, using \eqref{P(Q(M1)<)<}, {\it and\/}
\begin{equation}\label{key}
\log n_1\le s(\bold n)+\log(n_2-n_1), 
\end{equation}
we obtain
\begin{equation}\label{prelemma}
\begin{aligned}
&\pr\Bigl(\min_{\mathcal M}Q(\mathcal M)\le k\Bigr)\le P_k^-(\bold n)\binom{n_2}{n_1} n_1!=O\bigl(n_1\hat{\eps}(\bold n)\bigr)\\
&\quad\le_b\left\{\begin{aligned}
&\exp\left(-0.48(n_2-n_1)s(\bold n)^{2(1-b)}\right),&&\text{if }s(\bold n)\to\infty,\\
&\exp\left(-\theta\bigl(n_1^{1-2a}\bigr)\right),&&\text{if }s(\bold n)=O(1),\end{aligned}\right.
\end{aligned}
\end{equation}
provided that $2(1-b)>1$, i.e. $b<1/2$.\\

We have proved
\begin{Lemma}\label{P(Q(M)<)small} Let $b<1/2$ in the definition \eqref{delta(n)=} of $\delta(\bold n)$. Then
\[
\pr\Bigl(\min_{\mathcal M}Q(\mathcal M)\le (1-\delta(\bold n))n_2s(\bold n)\Bigr)
\le\!\left\{\begin{aligned}
&\!\exp\Bigl(-\theta\bigl(s(\bold n)^{2(1-b)}\bigr)\Bigr),&&\text{if } s(\bold n)\to\infty,\\
&\!\exp\Bigl(-\theta\bigl(n_1^{1-2a}\bigr)\Bigr),&&\text{if } s(\bold n)=O(1).
\end{aligned}\right.
\]
\end{Lemma}

We are about to prove that, on the other hand,  w.h.p. 
$\max_{\mathcal M}Q(\mathcal M)\le k:=(1+\delta(\bold n))(n_2-1)s(\bold n)$. 

The argument runs parallel to the above proof of Lemma \ref{P(Q(M)<)small}. Analogously to 
\eqref{P(Q(M1)<)<} and \eqref{Pk^-simple}-\eqref{H(s,xi)}, we have
\begin{equation}\label{P(Q(M1)>)<}
\pr\Bigl(\max_{\mathcal M}Q(\mathcal M)\ge k\Bigr)\le \binom{n_2}{n_1}n_1! P^+_{k}(\bold n),\quad P^+_{k}(\bold n):=\sum_{\kappa\ge k}P_{\kappa}(\bold n),
\end{equation}
and
\begin{align}
P^+_{k}(\bold n)&\le_b\frac{1}{(n_1-1)!}\int_0^{n_1}\inf_{\xi\in [1,\infty)}\exp\bigl(H(s,\xi)\bigr)\,ds,
\label{Pksimple}
\end{align}
with $H(s,\xi)$ defined in \eqref{H(s,xi)}. So now the Chernoff parameter $\xi$ exceeds $1$. Introduce
$t=(1+\delta(\bold n))s(\bold n)$, and  set $\xi(s)=t/s$ for $s\le t$ and $\xi(s)=1$ for $s>t$. 

Let us bound $\int_0^{n_1}\exp(H(s,\xi(s))\,ds$. First, using \eqref{H(s,xi)},  we have 
\begin{align*}
&\int_t^{n_1}\!\!\exp\bigl(H(s,\xi(s)\bigr)\,ds=\int_t^{n_1}\!\!\exp\bigl(H(s,1)\bigr)\,ds\\
&\le \int_t^{n_1}\!\!e^{-s(n_2-n_1)}(1-e^{-s})^{n_1-1}\,ds.
\end{align*}
The same steps as in \eqref{intK(z)=}-\eqref{znotinz1,z2} deliver
\begin{equation}\label{int(t,n1)<}
\begin{aligned}
&\quad\int_t^{n_1}\!\!\exp\bigl(H(s,\xi(s)\bigr)\,ds\le\eps(\bold n)\frac{n_2}{n_1(n_2-n_1)}\,\binom{n_2}{n_1}^{-1},\\
\eps(\bold n)&=\left\{\begin{aligned}
&\exp\bigl(-0.99(n_2-n_1) s(\bold n)^{1-b}\bigr),&&\text{if }s(\bold n)\to\infty,\\
&\exp\bigl(-\theta(n_1^{1-2a})\bigr),&&\text{if }s(\bold n)=O(1).\end{aligned}\right.
\end{aligned}
\end{equation}
For $s(\bold n)=O(1)$ the factor $\eps(\bold n)$ is well suited for our needs. However, in the case $s(\bold n)\to \infty$ the power of $s(\bold n)$ is too low. Fortunately, modifying and extending the argument for Lemma \ref{P(Q(M)<)small},
we can double the power of $s(\bold n)$.

Let us write
\begin{align*}
&e^{-s(n_2-n_1)}(1-e^{-s})^{n_1-1}=e^{\psi_1(s)},\\
&\psi_1(s):= -s(n_2-n_1)+(n_1-1)\log\bigl(1-e^{-s}\bigr). 
\end{align*}
We have
\begin{align*}
\psi_1^{(1)}(s)&=-(n_2-n_1)+(n_1-1)\frac{1}{e^s-1},\\
\psi_1^{(2)}(s)&=-(n_1-1)\frac{e^s}{(e^s-1)^2},\\
\psi_1^{(3)}(s)&=\frac{2}{(e^s-1)^3}+\frac{3}{(e^s-1)^2}\le \frac{3e^s}{(e^s-1)^3},
\end{align*}
and $\psi_1^{(4)}(s)<0$. So $\psi_1(s)$ is concave, and attains its maximum at $s_1(\bold n)=\log\tfrac{n_2-1}{n_2-n_1}$. Clearly $\psi_1^{(1)}(s_1(\bold n))=0$, $s_1(\bold n)<s(\bold n)<t$, and
\begin{align}
\psi_1(s_1(\bold n))&=-\log\frac{(n_2-1)^{n_2-1}}{(n_1-1)^{n_1-1}(n_2-n_1)^{n_2-n_1}}\le -\log\binom{n_2-1}{n_1-1},\label{psi(s1(n)=}\\
\psi_1^{(2)}(s_1(\bold n))&\sim - \frac{n_2(n_2-n_1)}{n_1},\notag\\
\psi_1^{(3)}(s)&\le \psi^{(3)}(s_1(\bold n))\lesssim \frac{3(n_2-n_1)^2}{n_1^2}\sim 3s^{-2}(\bold n).\notag
\end{align}
Since $t-s_1(\bold n)\sim t-s(\bold n)=\delta(\bold n)s(\bold n)$, and $\tfrac{n_2(n_2-n_1)}{n_1}\sim (n_2-n_1)$, we have
\begin{equation}\label{psi(t)<}
\begin{aligned}
\psi_1(t)&\le \psi_1(s_1(\bold n))-\frac{n_2-n_1}{2}(\delta(\bold n)s(\bold n))^2(1+o(1))+O\bigl(\delta^3(\bold n)s(\bold n)\bigr)\\
&=\psi_1(s_1(\bold n))-\frac{n_2-n_1}{2}(\delta(\bold n)s(\bold n))^2(1+o(1)),
\end{aligned}
\end{equation}
and likewise
\[
\psi_1'(t)=-(n_2-n_1)\delta(\bold n)s(\bold n)(1+o(1))\to -\infty.
\]
By concavity of $\psi_1(s)$, we also have $\psi_1(s)\le \psi_1(t)+\psi_1'(t)(t-s)$. Therefore, as $\delta(\bold n)
=s(\bold n)^{-b}$ and $\tfrac{n_2}{n_2-1}=e^{s(\bold n)}$,
\begin{equation}\label{upgrade}
\begin{aligned}
&\int_t^{n_1}\!\!\exp\bigl(H(s,\xi(s)\bigr)\,ds\le e^{\psi_1(t)}\int_t^{\infty}\exp(\psi_1'(t)(s-t))\,ds
=\frac{e^{\psi_1(t)}}{-\psi_1'(t)}\\
&\le_b \frac{n_2}{n_1(n_2-n_1)}\binom{n_2}{n_1}^{-1}
\exp\left(\!-\frac{n_2-n_1}{2}s(\bold n)^{2(1-b)}(1+o(1))\!\right)\\
&\le_b\binom{n_2}{n_1}^{-1}\exp\left(\!-0.49(n_2-n_1)s(\bold n)^{2(1-b)}\!\right),\\
\end{aligned}
\end{equation}
if $b<1/2$. Thus the bound in the first equation from \eqref{int(t,n1)<} continues to hold 
for the much smaller
\begin{equation}\label{better,eps(n)}
\eps(\bold n):=\exp\bigl(-0.49(n_2-n_1)s^{2(1-b)}(\bold n)\bigr).
\end{equation}

Turn to $s\in [0,t]$. Let us continue with the case $s(\bold n)\to\infty$. Analogously to \eqref{t,oo},  we have 
\begin{equation}\label{int[0,t]<e(psi)}
\int_0^{t}\exp\bigl(H(s,\xi(s)\bigr)\,ds\le_b e^{\psi(t)}.
\end{equation}
We know that $\psi(s)$ attains its maximum at $s(\bold n)$, and $\psi(s(\bold n))\le -\log\binom{n_2}{n_1}$.
Since $t-s(\bold n)=\delta(\bold n)s(\bold n)$, just like \eqref{psi(t)<} we have
\begin{equation}\label{psi(t)<psi(sn)}
\psi(t)\le \psi(s(\bold n))-\frac{n_2-n_1}{2}(\delta(\bold n)s(\bold n))^2(1+o(1)).
\end{equation}
So
\begin{align}
\int_0^{t}\exp\bigl(H(s,\xi(s)\bigr)\,ds&\le _b \binom{n_2}{n_1}^{-1}
\exp\left(-0.49(n_2-n_1)s(\bold n)^{2(1-b)}\right).\label{sameup}
\end{align}
By \eqref{upgrade} and \eqref{sameup},
\begin{equation}\label{total,slarge}
\int_0^{n_1}\exp\bigl(H(s,\xi(s)\bigr)\,ds\le_b\binom{n_2}{n_1}^{-1}
\exp\left(-0.49(n_2-n_1)s(\bold n)^{2(1-b)}\right).
\end{equation}

For the case $s(\bold n)=O(1)$ we have $\delta(\bold n)=n_1^{-a}$, $a<1/2$. The bound \eqref{int[0,t]<e(psi)} still holds. Furthermore, using  
\[
\psi''(s(\bold n))=-\frac{n_2(n_2-n_1)}{n_1}=\theta(n_2^2n_1^{-1}),\quad s(\bold n)=\log\frac{n_2}{n_2-n_1}\ge
\frac{n_1}{n_2},
\]
we have
\begin{align*}
\psi(t)&\le \psi(s(\bold n))-\frac{\psi''(s(\bold n))}{2}(\delta(\bold n)s(\bold n))^2(1+o(1))\\
&\le \psi(s(\bold n))-\theta\bigl(n_1^{1-2a}\bigr).
\end{align*}
So
\[
\int_0^{t}\exp\bigl(H(s,\xi(s)\bigr)\,ds\le \binom{n_2}{n_1}^{-1}\exp\Bigl(-\theta\bigl(n_1^{1-2a}\bigr)\Bigr).
\]
Combining this bound with \eqref{int(t,n1)<} ($s(\bold n)=O(1)$ case), we obtain
\begin{equation}\label{total,ssmall}
\int_0^{n_1}\exp\bigl(H(s,\xi(s)\bigr)\,ds\le_b\binom{n_2}{n_1}^{-1}\exp\Bigl(-\theta\bigl(n_1^{1-2a}\bigr)\Bigr).
\end{equation}
With the bounds \eqref{total,slarge} and \eqref{total,ssmall} at hand, we argue exactly like in \eqref{Pk^-<simple}, \eqref{key} and \eqref{prelemma} and establish
\begin{Lemma}\label{P(Q(M)>)large} In notations of Lemma \ref{P(Q(M)<)small},
\[
\pr\Bigl(\max_{\mathcal M}Q(\mathcal M)\ge (1+\delta(\bold n))n_2s(\bold n)\Bigr)
\le\!\left\{\begin{aligned}
&\!\exp\Bigl(-\theta\bigl(s(\bold n)^{2(1-b)}\bigr)\Bigr),&&\text{if } s(\bold n)\to\infty,\\
&\!\exp\Bigl(-\theta\bigl(n_1^{1-2a}\bigr)\Bigr),&&\text{if } s(\bold n)=O(1).
\end{aligned}\right.
\]
\end{Lemma}

\section{Proof of Theorem \ref{R(M)appr}}

We need to show that, for $n_2\le n_1^{3/2-d}$, ($d<1/2$), w.h.p. for all
stable $\mathcal M$'s the husbands' rank $R(\mathcal M)$ is asymptotic to 
\begin{equation}\label{k(bold n)=}
k=k(\bold n):=n_1^2\,f(s(\bold n)),\quad f(x):=\frac{e^{x}-1-x}{x(e^{x}-1)}.
\end{equation}

Similarly to \eqref{P(Q(M1)<)<},  for $k\ge n_1$ we have
\begin{align}
\pr\Bigl(\min_{\mathcal M}R(\mathcal M)\le k\Bigr)&\le \binom{n_2}{n_1}\,n_1!\sum_{\ell\le k}P_{\ell}(\bold n),
\label{P(minR<)}\\
\pr\Bigl(\max_{\mathcal M}R(\mathcal M)\ge k\Bigr)&\le \binom{n_2}{n_1}\,n_1!\sum_{\ell\ge k}P_{\ell}(\bold n),
\label{P(maxR>)}
\end{align}
where $P_{\ell}(\bold n)$ is the probability that a generic injection from $[n_1]$ to $[n_2]$ is stable, and the
husbands' rank is $\ell$. Denote the first sum and the second sum $\mathcal P^-(k)$ and $\mathcal P^+(k)$ respectively. Let us bound these probabilities for some $k=k^-(\bold n)$ and $k=k^+(\bold n)$
respectively, such that $k(\bold n)\in [k^-(\bold n),k^+(\bold n)]$.

By the formula for $P_{\ell}(\bold n)$ in Lemma \ref{extbasic}, we
bound the sums:
\begin{align}
\mathcal P^-(k)&\le\!\overbrace {\idotsint}^{2n_1}_{\bold x,\,\bold y\in [0,1]^{n_1}}
\inf\!\Big\{\!\Psi(\eta,\bold x,\bold y): \eta\le 1\!\Bigr\}\,d\bold x d\bold y,\label{Pn-(k)<Psi}\\
\mathcal P^+(k)&\le\!\overbrace {\idotsint}^{2n_1}_{\bold x,\,\bold y\in [0,1]^{n_1}}
\inf\!\Big\{\!\Psi(\eta,\bold x,\bold y): \eta\ge 1\!\Bigr\}\,d\bold x d\bold y,\label{Pn+(k)<Psi}\\
\Psi(\eta,\bold x,\bold y)&:=\eta^{n_1-k}\!\!\!\!\prod_{1\le i\neq j\le n_1}\!\!\!
\bigl(\bar x_i\bar y_j +x_i\bar y_j+\bar x_i y_j\eta\bigr)\cdot\prod_{h=1}^{n_1}\bar x_h^{n_2-n_1}.\label{Phi=}
\end{align}

{\bf (a)\/} Our first step is to dispense with the peripheral parts of the cube $[0,1]^{2n_1}$ whose contribution
to the integrals in \eqref{Pn-(k)<Psi} and \eqref{Pn+(k)<Psi} can be safely ignored. Fix 
$\gamma\in \bigl(\tfrac{1}{2},1)$, $\rho\in (0,1)$, and define
\begin{align*}
C_1&=\Bigl\{(\bold x,\bold y)\in [0,1]^{2n_1}: \sum_ix_iy_i\ge n_1^{\gamma}\Bigr\},\\
C_2&=\Bigl\{(\bold x,\bold y)\in [0,1]^{2n_1}: \sum_ix_i\ge n_1- n_1^{\rho}\Bigr\},\\
C_0&=[0,1]^{2n_1}\setminus(C_1\cup C_2).
\end{align*}
Denote by $I^{\pm}(C_{\alpha})$ the contributions of $C_{\alpha}$ to the value of the integrals
in \eqref{Pn-(k)<Psi} and \eqref{Pn+(k)<Psi}, $(\alpha=0,1,2)$. For $\alpha=1,2$, choosing $\xi=\eta=1$ we have
\[
I^{\pm}(C_{\alpha})\le \overbrace {\idotsint}^{2n_1}_{\bold x,\,\bold y\in C_{\alpha}}\prod_{1\le i\neq j\le n_1}(1-x_iy_j)\,d\bold x d\bold y.
\]
Here 
\begin{equation}\label{<n - sumsum}
\begin{aligned}
\prod_{h=1}^{n_1} \bar x_h^{n_2-n_1}&\le \exp\bigl(-(n_2-n_1)\sum_h x_h\bigr),\\
\log\prod_{1\le i\neq j\le n_1}\!\!\!\!(1-x_iy_j)&\le -\sum_{1\le i\neq j\le n_1}\!\!\!x_iy_j\le n_1-
\Bigl(\sum_ix_i\Bigr)\Bigl(\sum_j y_j\Bigr),
\end{aligned}
\end{equation}
and for $C_1$,  $s(\bold x):=\sum_i x_i\ge n_1^{\gamma}$, while by the Cauchy-Schwartz inequality and $x_i, y_j\le 1$,
\[
\Bigl(\sum_ix_i\Bigr)\Bigl(\sum_j y_j\Bigr)\ge \Bigl(\sum_i x_i^{1/2}y_i^{1/2}\Bigr)^2\ge
\Bigl(\sum_i x_iy_i\Bigr)^2\ge n_1^{2\gamma}.
\]
So 
\[
I^{\pm}(C_1)\le \exp\bigl(n_1-n_1^{2\gamma}-n_1^{\gamma}(n_2-n_1)\bigr),
\]
implying that the total contribution of $C_1$ to 
$\binom{n_2}{n_1}n_1!\,\mathcal P^{\pm}(k)$ is $\exp\bigl(-0.5 n_1^{2\gamma}\bigr)$, i.e. $o\bigl(e^{-cn_1}\bigr)$ for every $c>0$ as $n_2>n_1\to\infty$, since $2\gamma>1$. Using the definition of $C_2$ and the top inequality in \eqref{<n - sumsum}, and integrating innermost over $\bold y$, we have: with $s^*=n_1-n_1^{\rho}$,
\begin{align*}
I^{\pm}(C_2)&\le \exp\bigl(-0.5 n_1(n_2-n_1)\bigr)\!\!\overbrace {\idotsint}^{n_1}_{\bold x\in [0,1]^{n_1}:s(\bold x)\ge s^*}\!\!
\left(\int_0^1e^{-s(\bold x)}\,dy\!\!\right)^{n_1} d\bold x\\
&=\exp\bigl(-0.5 n_1(n_2-n_1)\bigr)\int_{s^*}^{n_1}\left(\frac{1-e^{-s}}{s}\right)^{n_1}\!\! f_{n_1}(s)\,ds,
\end{align*}
where $f_{n_1}(s)$ is the density of $\sum_iX_i$, the sum of the $n_1$ independent, $[0,1]$-uniform random variables. By Lemma \ref{densities}, 
\[
f_{n_1}(s)=\frac{s^{n_1-1}}{(n_1-1)!}\!\pr(L_{n_1}^+\le s^{-1}),
\]
where $L_{n_1}^+$ is the length of the longest subinterval in the partition of $[0,1]$ by $n_1-1$
independent, uniform points.  It was proved in \cite{Pit2} that 
\[
\pr\bigl(L_{n_1}^+\le\zeta\bigr)\le (n_1\zeta-1)^{n_1-1},\quad \forall\,\zeta>n_1^{-1}.
\]
Therefore
\begin{align*}
I^{\pm}(C_2)&\le\frac{\exp\bigl(-0.5 n_1(n_2-n_1)\bigr)}{(n_1-1)!}\left(\frac{n_1^{\rho}}{n_1-n_1^{\rho}}\right)^{n_1-1}\int_{s^*}^{n_1}s^{-1}\,ds\\
&=O\left(\frac{\exp\bigl(-0.5 n_1(n_2-n_1)-0.5(1-\rho)n_1\log n_1\bigr)}{(n_1-1)!}\right),
\end{align*}
implying that the total contribution of $C_2$ to $\binom{n_2}{n_1}n_1!\,
P^{\pm}(k)$ is \linebreak $o\bigl(e^{-cn_1}\bigr)$ for every $c>0$ as $n_2>n_1\to\infty$, since $\rho<1$.

{\bf (b)\/} Turn to the contribution of $C_0=[0,1]^{2n_1}\setminus \bigl(C_1\cup C_2\bigr)$. 
Let $(\bold x,\bold y)\in C_0$. Similarly to \eqref{<n - sumsum}, we obtain
\begin{equation}\label{logprod=}
\begin{aligned}
&\log\left[\prod_{ i\neq j}(\bar x_i\bar y_j+x_i\bar y_j+\bar x_iy_j\eta)\cdot\prod_h\bar x_h^{n_2-n_1}\right]\\
&\le\! \sum_{i\neq j}y_j\bigl((\eta-1)-x_i\eta\bigr)-(n_2-n_1)s\\
&=-s(\eta)\sum_jy_j-(n_2-n_1)s+\eta\sum_ix_iy_i,
\end{aligned}
\end{equation}
where $s=\sum_ix_i$, and
\begin{equation}\label{s(xi,eta)=}
s(\eta)=s-(n_1-1-s)(\eta-1)=n_1-1-\eta(n_1-1-s).
\end{equation}
By the definition of $C_0$, the last term in \eqref{logprod=} is $O\bigl(\eta\min(n_1^{\gamma},s)\bigr)$. So $\Psi(\cdot)$ in\eqref{Phi=} is bounded via
\begin{equation}\label{Phi<}
\Psi(\eta,\bold x,\bold y)\le \eta^{n_1-k}\exp\Bigl[-s(\eta)\sum_jy_j
-(n_2-n_1)s+O\bigl(\eta\min(n_1^{\gamma},s)\bigr)\Bigr].
\end{equation}
To upper-bound $I^{\pm}(C_0)$, we will choose $\eta=\eta^{\pm}$ dependent on
$s$ only. Of course, the admissible $\eta^{\pm}(s)$ need to satisfy the conditions
$\eta^+(s)\ge 1$ and $\eta^{-}(s)\le 1$, respectively. Whatever our choice will
be, integrating innermost with respect to $y$ and using the bound $f_{n_1}(s)\le \tfrac{s^{n_1-1}}{(n_1-1)!}$, we obtain 
\begin{equation}\label{Ipm<Jpm}
\begin{aligned}
I^{\pm}(C_0)&\le \frac{J^{\pm}(C_0)}{(n_1-1)!},\\
J^{\pm}(C_0)&=\int\limits_0^{s^*}\!\exp\Bigl(\!\mathcal H_k(s,\eta^{\pm})
+O\bigl(\eta^{\pm}\min(n_1^{\gamma},s)\bigr)\!\Bigr) ds,\\
\mathcal H_k(s,\eta)&:=(n_1-k)\log\eta+n_1\log\frac{1-\exp\bigl(-s(\eta)\bigr)}{s(\eta)}\\
&+(n_1-1)\log s-(n_2-n_1)s;
\end{aligned}
\end{equation}
here $s^*=n_1-n_1^{\rho}$, and $s(\eta)$ was defined in \eqref{s(xi,eta)=}. \\

An ideal $\eta^{\pm}(s)$ is an {\it admissible\/} $\eta(s,k^{\pm})$ that maximizes $\mathcal H_k(s,\eta)$,
($k=k^{\pm}(\bold n)$). (We hasten to add that the parameters $k^{\pm}(\bold n)$, that sandwich $k(\bold n)$, will be defined shortly.)

As in the proof of Theorem \ref{Q(M)appr}, we are content to choose $\eta^{\pm}(s)$ asymptotically close 
to a stationary point $\eta(s,k)$ of $H_k(s,\eta)$,  considered as a function of $\eta$, provided that
this point is admissible, of course. Now a stationary point is a  root of
\begin{align*}
\frac{\partial\mathcal H_k(s,\eta)}{\partial\eta}&=\frac{n_1-k}{\eta}+n_1f(s(\eta))(n_1-1-s)=0,\\
f(x)&:=\frac{e^x-1-x}{x(e^x-1)};
\end{align*}
or, since $\eta(n_1-1-s)=n_1-1-s(\eta)$,
\begin{equation}\label{s(eta)=root}
k=n_1 + n_1(n_1-1-s(\eta)) f(s(\eta)),
\end{equation}
if  $s(\eta)=o(n_1)$. (It was the RHS  in this formula that prompted us to come up
with $k(\bold n)=n_1^2 f(s(\bold n))$.) Since $f(x)$ decreases with $x$, we hope for $s(\eta)$ to be 
asymptotic to 
\[
s^{\pm}(\bold n)=\frac{s(n)}{1\pm\delta(\bold n)},
\]
for $=k^{\pm}(\bold n)$. Let us wait a bit more before we settle on the exact formulas for $k^{\pm}(\bold n)$.

As $s(\eta)=n_1-1-\eta(n_1-1-s)$, we are thus led--informally, needless to say--to
\[
\eta^{\pm}(s)=\frac{n_1-1-s^{\pm}}{n_1-1-s},
\]
for $s$ such that $\eta^+(s)\ge1$,  and $\eta^-(s)\le 1$, respectively; otherwise, we use the fall-back
choices $\eta^+(s)=1$ and $\eta^-(s)=1$, respectively. To summarize, 
we have defined
\begin{equation*}
\eta^-(s)=\!\left\{\!\begin{aligned}
&1,&& s\in [s^-,s^*],\\
&\frac{n_1-1-s^-}{n_1-1-s},&&s<s^-;\end{aligned}\right.
\quad
\eta^+(s)=\!\left\{\!\begin{aligned}
&\frac{n_1-1-s^+}{n_1-1-s},&& s\in [s^+,s^*],\\
&1,&&s<s^+.\end{aligned}\right.
\end{equation*}
 \\

{\bf The ``$+$'' case.\/}  Suppose that $s\in [s^+,s^*]$. Then $s(\eta^{+}(s))=s^+$, [see \eqref{s(xi,eta)=},
\eqref{Ipm<Jpm}], and therefore $\mathcal H_{k^+}(s,\eta^+(s))=H^+(s)$,
\begin{align*}
H^+(s)&:=(n_1-k^+)\log\frac{n_1-1-s^+}{n_1-1-s}+n_1\log\frac{1-e^{-s^+}}{s^+}\\
&\quad+(n_1-1)\log s-(n_2-n_1)s.
\end{align*}
The function $H^+(s)$ is concave and its derivative
\[
\frac{dH^+(s)}{ds}=\frac{n_1-k^+}{n_1-1-s}+\frac{n_1-1}{s}-(n_2-n_1)
\]
vanishes at $s=s^+$ if we {\it define\/} 
\begin{equation}\label{k^+=}
\begin{aligned}
&k^+(\bold n)=n_1+(n_1-1-s^+(\bold n))\left(\frac{n_1-1}{s^+(\bold n)}-(n_2-n_1)\right)\\
&=n_1+n_1(n_1-1-s^+(\bold n))\left(f(s(\bold n))+\frac{(1-n_1^{-1})\delta(\bold n)-n_1^{-1}}{s(\bold n)}\right).
\end{aligned}
\end{equation}
Now, by the definition of $\delta(\bold n)$,
\begin{equation}\label{delta^*=}
\begin{aligned}
\delta^*(\bold n)&:=\frac{\delta(\bold n)}{s(\bold n)f(s(\bold n))}
=\left\{\begin{aligned}
&O\bigl(s(\bold n)^{-b}\bigr),&&\text{if } s(\bold n)\to\infty,\\
&O\bigl(n_2n_1^{-1-a}\bigr),&&\text{if } s(\bold n)=O(1),\end{aligned}\right.
\end{aligned}
\end{equation}
where $b<1$ and $a<1/2$. Since $n_2\le n_1^{3/2-\sigma}$, this $n_2n_1^{-1-a}$ tends to zero,
provided that, in addition, $a>1/2-\sigma$, which we assume. So from \eqref{k^+=} we have 
\[
k^+(\bold n)\le \bigl(1+1.01\delta^*(\bold n)\bigr)n_1^2f(s(\bold n)).
\]
Since $d^2H^+/ds^2< -n_1/s^2$, we have then: for $s\ge s^+$,
\begin{align*}
H^+(s)&\le \psi_1(s^+) -\frac{n_1}{2}\left(\frac{s-s^+}{s}\right)^2,\\
\psi_1(s)&=(n_1-1)\log\bigl(1-e^{-s}\bigr)-(n_2-n_1)s.
\end{align*}
There is still the remainder term $O\bigl(\eta^{+}\min(n_1^{\gamma},s)\bigr)$ in \eqref{Ipm<Jpm}.
It is easy to check that, for $s\le s^*=n_1-n_1^{\rho}$, we have
 \[
 -\frac{n_1}{2}\left(\frac{s-s^+}{s}\right)^2+O\bigl(\eta^{+}\min(n_1^{\gamma},s)\bigr)
\le A s(\bold n),
\]
for some constant $A=A(\rho,\gamma)$, provided that $\rho>\gamma$. And this condition can be met, because up to now there were only two conditions on $\rho$ and $\gamma$: $\rho\in (0,1)$
and $\gamma\in (1/2,1)$. 
Only a minor modification of the argument from \eqref{psi(s1(n)=} to \eqref{psi(t)<} is needed to conclude
that 
\begin{equation}\label{s>s+}
\begin{aligned}
&\int\limits_{s^+}^{s^*}\!\exp\Bigl(\!\mathcal H_{k^+}(s,\eta^{+})
+O\bigl(\eta^{+}\min(n_1^{\gamma},s)\bigr)\!\Bigr) ds
\le n_1e^{\psi_1(s^+)+A s(\bold n)}\\
&\le\left\{\begin{aligned}
&\exp\left(\!-0.49(n_2-n_1)s(\bold n)^{2(1-b)}\right)\binom{n_2}{n_1}^{-1},&&\text{if } s(\bold n)\to\infty,\\
&\exp\left(-\theta\bigl(n_1^{1-2a}\bigr)\right)\binom{n_2}{n_1}^{-1},&&\text{if } s(\bold n)=O(1).\end{aligned}\right.
\end{aligned}
\end{equation}
Continuing with the ``$+$'' case,  suppose that $s\in (0,s^+)$. Then we have $\eta^+(s)=1$, so that
$s(\eta^{+}(s))=s$, and $\mathcal H_{k+}(s,1)=\psi_1(s)$. On $(0,s^+]$ the function $\psi_1(s)$ attains its maximum at $s^+$. So the contribution of this interval to $J^+(C_0)$  in \eqref{Ipm<Jpm} is at most the bottom bound in \eqref{s>s+}. Thus the total contribution of $C_0$ to $\binom{n_2}{n_1}n_1!\,\mathcal P^+(k^+)$ is, order-wise, below the same bound.
And we recall that the contribution of the peripheral domain $C_1\cup C_2$ is $e^{-c^*n_1}$, at most.
So, by \eqref{P(maxR>)}, we obtain
\begin{equation}\label{P(R(M)>)<}
\begin{aligned}
&\pr\Bigl(\max_{\mathcal M}R(\mathcal M)\ge \bigl(1+1.01\delta^*(\bold n)\bigr)n_1^2f(s(\bold n))\Bigr)\\
&\le\left\{\begin{aligned}
&\exp\left(\!-0.49(n_2-n_1)s(\bold n)^{2(1-b)}\right),&&\text{if } s(\bold n)\to\infty,\\
&\exp\left(-\theta\bigl(n_1^{1-2a}\bigr)\right),&&\text{if } s(\bold n)=O(1).
\end{aligned}\right.
\end{aligned}
\end{equation}

{\bf The ``$-$'' case.\/} Suppose that $s\in [s^-,s^*]$. Then $\eta^{-}(s)=1$, $s(\eta^-(s))=s$, and therefore
the contribution of $[s^-,s^*]$ to $J^-(C_0)$ is, at most,
\begin{equation*}
\int_{s^-}^{s^*}\!\!\exp\left(\psi_1(s)+O\bigr(\min(n_1^{\gamma},s)\bigr)\right) ds.
\end{equation*}
On $[s^-,s^*]$, $\psi_1(s)$ attains its maximum at $s^-=\tfrac{s(\bold n)}{1-\delta(\bold n)}$, and similarly to the ``$+$'' case, the integral is bounded by the bottom bound in \eqref{s>s+}.

Suppose finally that $s\in (0,s^-)$; this is where we will choose $k^-(\bold n)$. Then $s(\eta^{-}(s))=s^-$, and therefore
\begin{align*}
\mathcal H_{k^-}(s,\eta^-(s))&=H^{-}(s) + O(1),\\ 
H^{-}(s)&:=(n_1-k^-)\log\frac{n_1-1-s^-}{n_1-1-s}+n_1\log\frac{1-e^{-s^-}}{s^-}\\
&\quad+(n_1-1)\log s-(n_2-n_1)s.
\end{align*}
Here the derivative
\[
\frac{dH^{-}(s)}{ds}=\frac{n_1-k^-}{n_1-1-s}+\frac{n_1-1}{s}-(n_2-n_1)
\]
vanishes at $s^-$ if we define
\begin{equation}\label{k^-}
\begin{aligned}
&k^-(\bold n)=n_1+(n_1-1-s^-(\bold n))\left(\frac{n_1-1}{s^-(\bold n)}-(n_2-n_1)\right)\\
&=n_1+n_1(n_1-1-s^-(\bold n))\left(f(s(\bold n))-\frac{(1-n_1^{-1})\delta(\bold n)+n_1^{-1}}{s(\bold n)}\right).
\end{aligned}
\end{equation}
 So from \eqref{k^-} we have
\[
k^-(\bold n)>\bigl(1-1.01\delta^*(\bold n)\bigr)n_1^2f(s(\bold n)).
\]
Since $H^{-}(s)$ is concave and its derivative at $s^-$ is zero, it attains maximum at $s^-$, which
equals
\[
n_1\log\frac{1-e^{-s^-}}{s^-}+(n_1-1)\log s^{-}-(n_2-n_1)s^-\le \psi_1(s^-).
\] 
Therefore the contribution of $[0,s^-]$ to $J^-(C_0)$ is below 
the bottom RHS in \eqref{s>s+}.  

We conclude that
\begin{equation}\label{P(R(M)<)<}
\begin{aligned}
&\pr\Bigl(\min_{\mathcal M}R(\mathcal M)\le \bigl(1-1.01\delta^*(\bold n)\bigr)n_1^2f(s(\bold n))\Bigr)\\
&\le\left\{\begin{aligned}
&\exp\left(\!-0.49(n_2-n_1)s(\bold n)^{2(1-b)}\right),&&\text{if } s(\bold n)\to\infty,\\
&\exp\left(-\theta\bigl(n_1^{1-2a}\bigr)\right),&&\text{if } s(\bold n)=O(1).
\end{aligned}\right.
\end{aligned}
\end{equation}
Combining \eqref{P(R(M)<)<} and \eqref{P(R(M)>)<}, we complete the proof of Theorem \ref{R(M)appr}.

\section{Rotations exposed in random stable matchings}  For $n_1=n_2$, Irving and Leather 
\cite{IrvLea} proved  the following deep result. For every 
stable matching $\mathcal M$ different from the men-optimal stable matching $\mathcal M_1$ 
there exists a sequence of stable matchings $\{\mathcal M^{(j)}\}_{1\le j\le t}$ with $\mathcal M^{(1)}=\mathcal M_1$ and $\mathcal M^{(t)}=\mathcal M$ such that each $\mathcal M^{(j+1)}$ is obtained
from $\mathcal M^{(j)}$ via a rotation step. 
It involves a cyclically ordered sequence of pairs $(m_i,w_i)$, $i\in [r]$, matched in $\mathcal M^{(j)}$, such that each woman $w_{i+1}$ is the best choice for the man
$m_i$ among women to whom he prefers his wife $w_i$, and who prefer $m_i$ to their husbands in $\mathcal M_j$. Pairing each $m_i$ with $w_{i+1}$  we obtain the next stable matching $M^{(j+1)}$, in which each woman $w_i$ gets a better husband $m_{i-1}$, and all other women keep their husbands  
unchanged. As Rob Irving  pointed out \cite{Irv} this theorem holds for the case of
$n_1\neq n_2$ as well. Thus, once we bound the expected total length of the rotations in all
the stable matchings, we will obtain an upper bound for the expected number of all members with more than one stable partner.

Let $(\mathcal M,\mathcal M')$ be a given pair of matchings, with the same set of $n_1$ women,
such that $\mathcal M'$ is obtained from $\mathcal M$ by breaking up some pairs $(m_1,w_1),\dots,
(m_r,w_r)$ in $\mathcal M$ and pairing $m_i$ with $w_{i+1}$, ($w_{r+1}:=w_1$).  Let $A$ denote the
event that $\mathcal M$ is stable and $\{(m_i,w_i)\}$ is a rotation in $\mathcal M$, so that
$\mathcal M'$ is stable as well. By symmetry, $\pr(A)$ depends only on $\bold n=(n_1,n_2)$, so we
denote it $P(\bold n,r)$.

\begin{Lemma}\label{P(nr)=} 
\[
P(\bold n,r)=\!\overbrace {\idotsint}^{2n_1}_{\bold x,\,\bold y\in [0,1]^{n_1}}\prod_{k=1}^rx_ky_k\,\prod_{ i\neq j}(1-x_iy_j)\prod_h\bar x_h^{n_2-n_1}\,d\bold x d\bold y;
\]
in the second product $j\neq i+1(\text{mod }r)$ for $1\le i\le r$.
\end{Lemma}

\begin{proof} For $n_1=n_2$ this is Lemma 3.2 (a) in \cite{Pit2}.
Here $\mathcal M$ is the identity mapping from the men set $[n_1]$ to the
women set $[n_1]$, and the rotation candidate is formed by the (man, woman) pairs  $(1,1),\dots, (r,r)$. The
integrand is the probability that $\mathcal M$ is stable and that the sequence $(1,1),\dots, (r,r)$ is
indeed a rotation exposed in $\mathcal M$, {\it conditioned\/} on the event
\begin{align*}
&X_{i,i+1}=x_i,\,\,(1\le i\le r),\,(X_{r,r+1}:=X_{r,1});\quad X_{i,i}=x_i,\,\,r+1\le i\le n_1,\\
&\qquad\qquad\qquad\qquad\quad Y_{j,j}=y_j,\,\, (1\le j\le n_1).
\end{align*}
\end{proof}
\begin{Corollary}\label{P(n,r)<,1} Denoting $s=\sum_{i\in [n_1]}x_i$, and
\[
E_1(u):=\frac{1-e^{-u}}{u},\quad E_2(u):=\frac{1-e^{-u}(1+u)}{u^2},
\]
we have 
\begin{align}
P(\bold n,r)&\le_b \overbrace {\idotsint}^{n_1}_{\bold x\in [0,1]^{n_1}}E_2(s)^r\,E_1(s)^{n_1-r}
e^{-(n_2-n_1)s}\left(\prod_{i=1}^r x_i\right)d\bold x\notag\\
&=\frac{1}{(n_1+r-1)!}\int_0^{n_1}E_2(s)^r\,E_1(s)^{n_1-r}
e^{-(n_2-n_1)s}s^{n_1+r-1}\,ds\notag\\
&=\frac{1}{(n_1+r-1)!}\int_0^{n_1}\! s^{-1}\!\left(\!\frac{e^s-1-s}{e^s-1}\!\right)^r \!(1-e^{-s})^{n_1}e^{-(n_2-n_1)s}\,ds.
\label{P(n,r)<int}
\end{align}
\end{Corollary}
\begin{proof} The proof  mimics the derivation of the bound (5.8), and its subsequent transformation into a univariate integral, in \cite{Pit2}.
\end{proof} 

The second factor in the integrand is below $1$. So from the proof of Theorem \ref{E(S(n))sim} in Section \ref{P(n)close} the integral in \eqref{P(n,r)<int} is asymptotically at most 
\begin{equation}\label{intsim}
s(\bold n)^{-1}\int_0^{n_1}(1-e^{-s})^{n_1}e^{-(n_2-n_1)s}\,ds=\frac{1+o(1)}{s(\bold n)(n_2-n_1)}\binom{n_2}{n_1}^{-1},
\end{equation}
$s(\bold n)=\log\frac{n_2}{n_2-n_1}$. Now the total number of injections of $[n_1]$ into $[n_2]$ is $\binom{n_2}{n_1}n_1!$, and the total 
number of cyclic sequences of $r$ matched pairs in an injection is $\binom{n_1}{r}(r-1)!=\tfrac{(n_1)_r}{r}$. Let $\mathcal R(\bold n)$ be the total {\it length\/} of all the rotations exposed in all the stable matchings. By \eqref{P(n,r)<int} and \eqref{intsim}, we have then
\begin{align*}
&\ex\bigl[\mathcal R(\bold n)\bigr]\le_b \frac{1}{s(\bold n)(n_2-n_1)}\sum_{r\ge 2}\frac{n_1!(n_1)_r}{(n_1+r-1)!}\\
&=\frac{n_1}{s(\bold n)(n_2-n_1)}\sum_{r\ge 2}\frac{(n_1)_r}{(n_1+r-1)_{r-1}}\\
&\le\frac{n_1}{s(\bold n)(n_2-n_1)}\sum_{r\ge 2}\exp\left(-\frac{r^2}{n_1}\right)
\le_b \frac{n_1^{3/2}}{s(\bold n)(n_2-n_1)}.
\end{align*}
Now $n_1^{-1}R(\bold n)$ is certainly an upper bound for both $m(\bold n)$, the {\it fraction\/}  of men, and 
$w(\bold n)$, the {\it fraction\/} of women, with more than one stable partner. It is easy to check that,
given $n_1$, the denominator $s(\bold n)(n_2-n_1)$ is strictly increasing with $n_2$, and that
\[
n_2-n_1=\bigl[n_1^{1/2} (\log n_1)^{-\gamma}\bigr],\,\,(\gamma<1)\Longrightarrow\frac{n_1^{1/2}}{s(\bold n)(n_2-n_1)}\to 0.
\]
Thus we proved
\begin{Theorem}\label{frctns,1} If $n_2-n_1\ge n_1^{1/2} (\log n_1)^{-\gamma}$, $(\gamma<1)$, then
$\lim m(\bold n) =\lim w(\bold n)=0$, in probability.
\end{Theorem}
\bi
{\bf Aknowledgment.\/} I owe debt of genuine gratitude to Jennifer Chayes for suggesting that the breakthrough results  in \cite{AshKanLes} might warrant a follow-up research.
I am grateful to Itai Ashlagi and Yash Kanoria for their encouraging interest
in this work. I thank Rob Irving for his patient explanation of why the fruitful notion of rotations
survives the transition from the classic case $n_1=n_2$ to the more general case $n_1\neq n_2$.

\end{document}